\documentclass[12pt]{article}
\usepackage{stmaryrd}
\usepackage[a4paper,includeheadfoot,margin=2.54cm,headheight=15pt]{geometry} 
\usepackage{amsmath,amsfonts,amssymb}
\usepackage{bm}
\usepackage{color}
\usepackage[english]{babel}
\usepackage{graphicx}
\usepackage{todonotes}
\usepackage[font={small},labelfont=bf,format=hang,format=plain,margin=0pt,width=0.8\textwidth]{caption}
\usepackage[list=true]{subcaption}
\numberwithin{equation}{section}

\DeclareMathOperator{\supp}{supp}
\DeclareMathOperator{\Span}{span}

\newcommand{\tn}{|\kern-1pt|\kern-1pt|}

\newcommand{\mcN}{\mathcal{N}}
\newcommand{\mcF}{\mathcal{F}}
\newcommand{\mcT}{\mathcal{T}}

\newcommand{\IR}{\mathbb{R}}
\newtheorem{lem}{Lemma}[section]

\newtheorem{thm}{Theorem}[section]
\newtheorem{rem}{Remark}[section]

\newenvironment{proof}{\noindent \newline {\bf Proof.}}
{\hfill \mbox{\fbox{} } \newline}

\title{\bf CutIGA with Basis Function Removal \thanks{This research was supported in part by the Swedish Foundation for Strategic Research Grant No.\ AM13-0029, the Swedish Research Council Grants Nos.\  2013-4708, 2017-03911, and the Swedish Research Programme Essence}}

\date{\today}

\author{
Daniel Elfverson\footnote{Department of Mathematics and Mathematical Statistics, Ume{\aa} University, SE-90187 Ume{\aa}, Sweden} 
\mbox{ }  
Mats G. Larson\footnote{Department of Mathematics and Mathematical Statistics, Ume{\aa} University, SE-90187 Ume{\aa}, Sweden} 
\mbox{ }  
Karl Larsson\footnote{Department of Mathematics and Mathematical Statistics, Ume{\aa} University, SE-90187 Ume{\aa}, Sweden} 
}
\date{\today}

\begin{document}
\maketitle
\begin{abstract}
We consider a cut isogeometric method, where the boundary of the domain is allowed to cut 
through the background mesh in an arbitrary fashion for a second order elliptic model 
problem. In order to stabilize the method on the cut boundary we remove basis 
functions which have small intersection with the computational domain. We determine criteria 
on the intersection which guarantee that the order of convergence in the energy norm is not affected by the removal. The higher order 
regularity of the B-spline basis functions leads to improved 
bounds compared to standard Lagrange elements.

\end{abstract}

\section{Introduction}

\paragraph{Background and Earlier Work.}
CutFEM and CutIGA, are methods where the geometry of the 
domain is allowed to cut through the background mesh in an arbitrary fashion, which manufactures so called 
cut elements at the boundary. This approach typically leads to some loss of stability and ill conditioning of the resulting stiffness matrix that must be handled in some way. Several approaches have 
been proposed:
\begin{itemize}
\item Gradient jump penalties or some related stabilization term, see \cite{Bu10} and 
\cite{BuCl15}. 
\item Adding a small amount of extra stiffness to each active element as is done in the finite cell method, see \cite{DaDu15} and \cite{PaDu07}. 
\item Element merging where small elements are associated with a neighbor which has 
a large intersection. For DG methods see \cite{JoLa2013} 
and for CG methods see \cite{BaVe17}. 
\item Basis function removal where basis functions with support that has a small 
intersection with the domain are removed. For the case of isogeometric spline spaces 
see \cite{EmDo10}.
\end{itemize}
For a general introduction to CutFEM we refer to the overview 
paper \cite{BuCl15} and for an introduction to isogeometric analysis we refer to \cite{CoHu09}.

\paragraph{New Contributions.} We investigate the basis function removal approach based on simply eliminating basis functions that 
has a small intersection with the domain in the context of isogeometric analysis, more precisely we employ B-spline spaces 
of order $p$ with maximal regularity $C^{p-1}$. To this end we need 
to make the meaning of small intersection precise and our guideline will be that we should not lose order in a given norm. In particular, we consider the the error in the energy norm and show that we may remove basis functions with sufficiently small energy norm and 
still retain optimal order convergence.


We also quantify the meaning of a basis function with sufficiently small energy norm in terms of the size of the intersection between 
the support of the basis function and the domain. 
In order to measure the size of the intersection we consider a 
corner inside the domain and we let $\delta_i$, $i=1,\dots,d$ with 
$d$ the dimension, be the distance from the corner to the intersection of edge $E_i$ with the boundary. If there is no intersection 
$\delta_i = h$. We then identify a condition on $\delta_i$ 
in terms of the mesh parameter $h$ which guarantees that we have optimal order convergence in the energy norm. The energy norm of the 
basis functions may be approximated by the diagonal element of the 
stiffness matrix and we propose a convenient selection procedure 
based on the diagonal elements in the stiffness matrix which is 
easy to implement. 

We also derive the condition on $\delta_i$ corresponding to 
the $W^1_\infty$ norm, which will be tighter since the norm is stronger and here we also need the continuity of the derivative 
of the basis functions. We discuss the approach in the context of standard Lagrange basis function where we note that we get much a tighter condition on $\delta_i$ in the energy norm and in the 
$W^1_\infty$ norm we find that it is not possible to remove 
basis functions.

We impose Dirichlet conditions weakly using nonsymmetric 
Nitsche, which is coercive by definition. Since the energy 
norm used in the nonsymmetric Nitsche method does not control 
the normal gradient on the Dirichlet boundary we do however 
need to add a standard least squares stabilization term on 
the elements in the vicinity of the boundary. Note that this 
term is element wise in contrast to the stabilization terms 
usually used in CutFEM.

When symmetric Nitsche is used to enforce Dirichlet boundary conditions stabilization appears to be necessary to guarantee 
that a certain inverse estimate holds. This bound is not improved by 
the higher regularity of the splines and will not be enforced 
in a satisfactory manner by basis function removal. 

\paragraph{Outline:} In Section 2 we introduce the model problem and the method, in Chapter 3 we derive properties of the bilinear form, 
define the interpolation operator, define the criteria for basis 
function removal, derive error bounds, and quantify $\delta$ in terms of $h$ for various norms, and finally in Section 4 we  present some illustrating numerical examples.

\section{The Model Problem and Method}

\subsection{Model Problem}
Let $\Omega$ be a domain in $\IR^d$ with smooth boundary $\partial \Omega$ 
and consider the problem: find $u: \Omega \rightarrow \IR$ such that 
\begin{alignat}{3}
-\Delta u &=f & \qquad &\text{in $\Omega$}
\\
u &=g_N & \qquad &\text{on  $\partial \Omega_N$}
\\
n\cdot \nabla u &=g_D & \qquad &\text{on  $\partial \Omega_D$}
\end{alignat}
For sufficiently regular data there exists a unique solution to this problem and we will be interested in higher order methods and therefore we will always assume that the solution satisfies the regularity estimate 
\begin{equation}\label{eq:regularity}
\| u \|_{H^{s}(\Omega)} \lesssim 1
\end{equation}
for $s\geq 2$. Here and below $a \lesssim b$ means that there is 
a positive constant $C$ such that $a \leq C b$.


\subsection{The Finite Element Method}

\begin{figure}
\centering
\subcaptionbox{$C^1 Q^2(\IR)$}{
\includegraphics[width=0.9\linewidth]{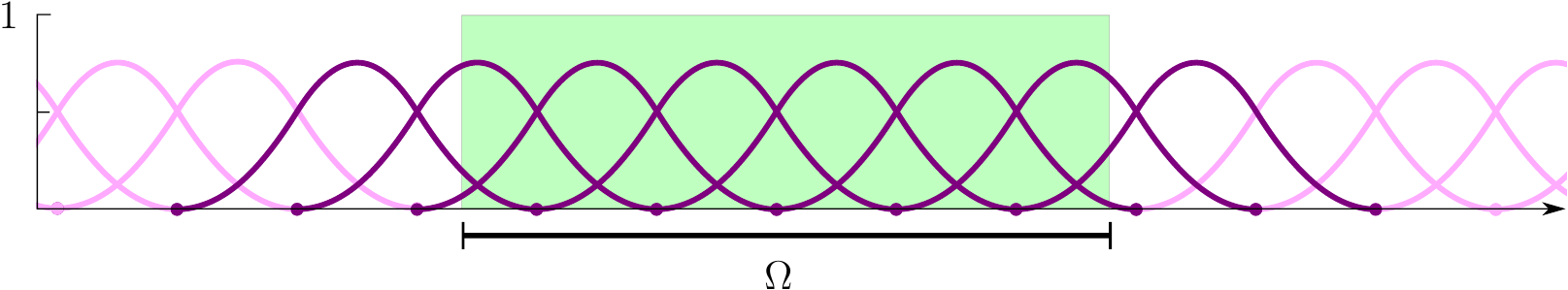}
}
\subcaptionbox{$C^2 Q^3(\IR)$}{
\includegraphics[width=0.9\linewidth]{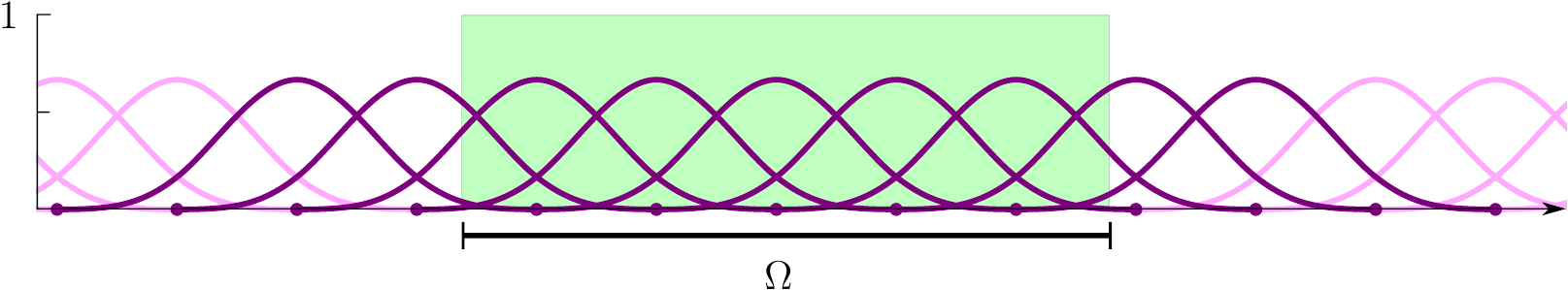}
}
\caption{B-spline basis functions in one dimension. The set $B$ of basis functions with non-empty support in $\Omega$ are indicated in deep purple. Note that basis functions crossing the boundary of $\Omega$ are defined analogously to interior basis functions.}
\label{fig:splines}
\end{figure}

\paragraph{The B-Spline Spaces.} 
\begin{itemize}
\item Let $\widetilde{\mcT}_{h}$, $h \in (0,h_0]$, be a family 
of uniform tensor product meshes in $\IR^d$ with mesh parameter 
$h$. 
\item Let $\widetilde{V}_{h}=C^{p-1}Q^p(\IR^d)$ be the space of 
$C^{p-1}$ tensor product B-splines of order $p$ defined on $\widetilde{\mcT}_{h}$. Let $\widetilde{B} = \{\varphi_i\}_{ i\in \widetilde{I}}$ be the standard basis in 
$\widetilde{V}_h$, where $\widetilde{I}$ is an index set.
\item Let $B = \{ \varphi \in \widetilde{B} \, : \,
\supp(\varphi) \cap \Omega \neq \emptyset \}$ be the set of basis functions with support that intersects $\Omega$. Let $I$ be an 
index set for $B$. Let $V_h = \Span\{B\}$ and let $\mcT_h = \{ T \in \widetilde{\mcT}_h : T \subset \cup_{\varphi \in B } \supp(\varphi) \}$. An illustration of the basis functions in $1D$ is given in Figure~\ref{fig:splines}.
\item Let $B = B_a \cup B_r$ be a partition into a set $B_a$ of active basis functions which we keep and a set $B_r$ of basis functions 
which we remove. Let $I = I_a \cup I_r$ 
be the corresponding partition of the index set. Let $V_{h,a} = \Span\{B_a\}$ be the active finite element space.
\end{itemize}
\begin{rem} To construct the basis functions in $\widetilde{V}_h$ 
we start with the one dimensional line $\IR$ and define a uniform partition, with nodes $x_i = i h$, $i \in \mathbb{Z}$, where $h$ 
is the mesh parameter, and elements $I_i = [x_{i-1},x_i)$. We 
define 
\begin{equation}
\varphi_{i,0}(x) 
=
\begin{cases}
1 & x \in I_i
\\
0 & x \in \IR \setminus I_i
\end{cases}
\end{equation}
The basis functions $\varphi_{i,p}$ are then defined by the 
Cox-de Boor recursion formula 
\begin{equation}
\varphi_{i,p} = \frac{x - x_i}{x_{i+p} - x_i} \varphi_{i,p-1}(x) 
+ \frac{x_{i+p+1} - x}{x_{i+p+1} - x_{i+1}} \varphi_{i+1,p-1}(x) 
\end{equation}
we note that these basis functions are $C^{p-1}$ and supported 
on $[x_i,x_{i+p+1}]$ which corresponds to $p+1$ elements, see Figure~\ref{fig:splines}.
We then define tensor product basis functions in $\IR^d$ of the 
form 
\begin{equation}
\varphi_{i_1,\dots,i_d} (x) = \prod_{k=1}^d \varphi_{i_k}(x_k)
\end{equation}
\end{rem}


\paragraph{The Nonsymmetric Method.} Find 
$u_{h,a} \in V_{h,a}$ such that 
\begin{equation}\label{eq:method-nonsym}
A_{h}(u_{h,a},v) = L_h(v)\qquad v \in V_{h,a}
\end{equation}
The forms are defined by
\begin{align}\label{eq:Ah}
A_h(v,w) &= a_h(v,w) + \tau h^2(\Delta v,\Delta w)_{\mcT_{h,D} \cap \Omega}
\\ \label{eq:Lh}
L_h(v) &= l_h(v) + \tau h^2 (f,\Delta v)_{\mcT_{h,D}\cap \Omega}
\end{align}
where 
\begin{align}
\label{eq:ah}
a_h(v,w) &= (\nabla v,\nabla w)_\Omega 
- (n\cdot \nabla v,w)_{\partial \Omega_D} 
+ (v, n\cdot \nabla w)_{\partial \Omega_D} 
+ \beta h^{-1} (v,w)_{\partial \Omega_D}
\\ \label{eq:lh}
l_h(v)&= (f, v)_\Omega 
+ (g_N,v)_{\partial \Omega_N}
+ (g_D,n\cdot \nabla v)_{\partial\Omega_D} 
+ \beta h^{-1}(g_D,v)_{\partial \Omega_D}  
\end{align}
with positive parameters $\beta$ and $\tau$. Furthermore, 
we used the notation
\begin{equation}
(v,w)_{\mcT_{h,D} \cap \Omega} 
= \sum_{T \in \mcT_{h,D}} (v,w)_{T\cap \Omega}
\end{equation}
where $\mcT_{h,D} \subset \mcT_h$ is defined by 
\begin{equation}\label{def:ThD}
\mcT_{h,D} = \mcT_h(U_\delta (\partial \Omega_D))
= \{T \in \mcT_h : T \cap U_\delta (\partial \Omega_D) \neq \emptyset \}  
\end{equation}
and
\begin{equation}\label{eq:Udelta}
U_\delta (\partial \Omega_D) = \left(\bigcup_{x\in \partial \Omega_D} B_\delta(x)\right) 
\cap \Omega
\end{equation}
with $\delta \sim h$ and $B_\delta(x)$ the open ball with center $x$ and 
radius $\delta$. We note that it follows from (\ref{def:ThD}) that $U_\delta(\partial \Omega_D) \subset \mcT_{h,D}$. 

\paragraph{Galerkin Orthogonality.} It holds 
\begin{equation}\label{eq:galort}
A_h( u - u_h, v) = 0 \qquad \forall v \in V_h
\end{equation}

\begin{rem}
In practice, $\mcT_{h,D}$ 
may be taken as the set of all elements that intersect the Dirichlet boundary $\partial \Omega_D$ and their neighbors, i.e. $\mcT_{h,D} = \mcN_h(\mcT_h(\partial \Omega_D))$.
\end{rem}

\begin{rem}{\bf (The Symmetric Method)} The symmetric version 
of (\ref{eq:method-nonsym}) takes the form: find $u_{h,a} \in V_{h,a}$ such that 
\begin{equation}
a_{h,\mathrm{sym}}(u_{h,a},v) + s_{h,\mathrm{sym}}(u_{h,a},v) = l_{h,\mathrm{sym}}(v)\qquad v \in V_{h,a}
\end{equation}
The forms are defined by
\begin{align}
\label{eq:ah-sym}
a_{h,\mathrm{sym}}(v,w) &= (\nabla v,\nabla w)_\Omega 
- (n\cdot \nabla v,w)_{\partial \Omega_D} 
- (v, n\cdot \nabla w)_{\partial \Omega_D} 
+ \beta h^{-1} (v,w)_{\partial \Omega_D}
\\ \label{eq:sh}
s_{h,\mathrm{sym}}(v,w) & = \gamma h^{2p - 1} ([D^p_{n_F} v ], [D^p_{n_F} w ])_{\mcF_{D,h}}
\\ \label{eq:lh-sym}
l_{h,\mathrm{sym}}(v)
&= (f,v)_\Omega + (g_N,v)_{\partial \Omega_N}
- (g_D,n\cdot \nabla v)_{\partial\Omega_D} + \beta h^{-1}(g_D,v)_{\partial \Omega_D}  
\end{align}
where $\beta$ and $\gamma$ are positive parameters, $\mcF_{h,D}$ 
is the set of interior faces which belong to an element in 
$\mcT_h(\partial \Omega_D)$, and $D_{n_F} = n_F \cdot \nabla$ is the 
directional derivative normal to the face $F$. 

The stabilization term $s_{h,\mathrm{sym}}$ provides the control 
\begin{equation}\label{eq:stab-control}
\|\nabla v \|^2_{\mcT_h(\partial \Omega_D)} \lesssim \| \nabla v \|^2_{\Omega} + \| v \|^2_{s_{h,\mathrm{sym}}}
\qquad v \in V_{h}
\end{equation}
where we note that we indeed obtain control on the full elements $T \in \mcT_{h}(\partial \Omega_D)$. The control (\ref{eq:stab-control}) is employed in the proof 
of the coercivity of $A_h$ in the symmetric case. More precisely, (\ref{eq:stab-control}) is used as follows 
\begin{equation}
h \| n \cdot \nabla v \|^2_{\partial \Omega_D} 
\lesssim 
\|\nabla v \|^2_{\mcT_h(\partial \Omega_D)} 
\lesssim  \| \nabla v \|^2_{\Omega} + \| v \|^2_{s_{h,\mathrm{sym}}}
\end{equation} 
where we used an inverse inequality in the first estimate 

In the symmetric formulation we stabilize to ensure that 
coercivity holds and this stabilization also implies that the resulting linear system of equations is well conditioned. Therefore, in the symmetric case, we do not employ basis function removal on the Dirichlet boundary.
\end{rem}

\section{Error Estimates}

\subsection{Basic Properties of $\boldsymbol{A}_{\boldsymbol{h}}$}
\paragraph{Energy norm.}
Define the norms 
\begin{align}\label{eq:energy-norm}
\tn v \tn_h^2 &= \| \nabla v \|^2_{\Omega} 
+ h^{-1}\| v \|^2_{\partial \Omega_D} 
+ \tau h^2 \| \Delta v  \|^2_{\mcT_{h,D} \cap \Omega}
\\
 \label{eq:energy-norm-star}
 \tn v \tn_{h,\bigstar}^2 &= \| \nabla v \|^2_{\Omega} 
 + h^{-1}\| v \|^2_{\partial \Omega_D} 
 + \tau h^2 \| \Delta v  \|^2_{\mcT_{h,D} \cap \Omega}
 + h \| n \cdot \nabla v \|^2_{\partial \Omega_D}
\end{align}

\paragraph{Coercivity.} For $\beta>0$ the form $A_h$ is coercive 
\begin{equation}\label{eq:coercivity}
\tn v \tn_{h}^2 \lesssim A_h(v,v) \qquad v \in V + V_h 
\end{equation}
where $V = H^2(\Omega)$. This result follows directly from 
the definition and the fact that the parameters $\tau\geq 0$ 
and $\beta>0$.

\paragraph{Continuity.} The form $A_h$ is continuous 
\begin{equation}\label{eq:continuity}
A_h(v,w) \lesssim \tn v \tn_h
\tn w \tn_{h,\bigstar} \qquad v,w \in V + V_h 
\end{equation}

\begin{proof} First we note that 
\begin{align}
A_h(v,w) &= (\nabla v, \nabla w)_\Omega 
- (n\cdot \nabla v, w)_{\partial \Omega_D} 
\\ \nonumber 
&\qquad 
+ (v, n \cdot \nabla w )_{\partial \Omega_D}
+ \beta h^{-1} (v, w )_{\partial \Omega_D}
+ \tau h^2 (\Delta v, \Delta w )_{\mcT_{h,D} \cap \Omega}
\\ \label{eq:continuity-b}
&\lesssim |(\nabla v, \nabla w)_\Omega 
- (n\cdot \nabla v, w)_{\partial \Omega_D}|
+ \tn v\tn_h \tn w \tn_{h,\bigstar}
\end{align}
We proceed with an estimate of the first term on the right hand side. 
To that end let $\chi:\Omega \rightarrow [0,1]$ be a smooth function 
such that 
\begin{equation}
\begin{cases}
\chi = 1 \quad \text{on $\partial \Omega_D$}
\\
\supp (\chi ) \subset U_\delta(\partial \Omega_D ) 
\\ 
\| \nabla \chi \|_{L^\infty(U_\delta (\partial \Omega_D ))} \lesssim \delta^{-1}
\end{cases}
\end{equation}
where $U_\delta(\partial \Omega_\delta)$ is defined in (\ref{eq:Udelta}).
Splitting the term $(\nabla v, \nabla w)_\Omega$ using $\chi$ and then applying Green's formula for the term in the vicinity of $\partial \Omega_D$ followed 
by some obvious bounds give
\begin{align} \nonumber
&(\nabla v, \nabla w )_{\Omega} - (n\cdot \nabla v, w)_{\partial \Omega_D}
\\
&\qquad 
= (\nabla v, (1 - \chi) \nabla w )_{\Omega} 
+ (\nabla v, \chi \nabla w )_{\Omega} 
- (n\cdot \nabla v, \chi w)_{\partial \Omega_D}
\\
&\qquad 
= (\nabla v, (1 - \chi) \nabla w )_{\Omega} 
- (\nabla \cdot (\chi \nabla v) , w )_{\Omega} 
\\
&\qquad 
= (\nabla v, (1 - \chi) \nabla w )_{\Omega} 
- (\nabla \chi \cdot \nabla v , w )_{\Omega} 
- (\chi \Delta v , w )_{\Omega}
\\
&\qquad \lesssim  
\|\nabla v\|_\Omega \|\nabla w \|_{\Omega} 
+  \delta^{-1} \|\nabla v\|_{U_\delta(\partial \Omega_D)} 
\|w \|_{U_\delta (\partial \Omega_D)} 
+ \| \Delta v\|_{U_\delta(\partial \Omega_D)} \|w \|_{U_\delta(\partial \Omega_D)}
\end{align}
Next using the bound 
\begin{equation}
 \|w \|^2_{U_\delta(\partial \Omega_D)} 
 \lesssim 
\delta \| w \|^2_{\partial \Omega_D} + \delta^2 \| \nabla w \|^2_{U_\delta(\partial \Omega_D)}
\end{equation}
see \cite{BuHa18}, we conclude that 
\begin{align}
&(\nabla v, \nabla w )_{\Omega} 
 - (n\cdot \nabla v, w)_{\partial \Omega_D}
\\
&\qquad \lesssim 
 \|\nabla v\|_\Omega \|\nabla w \|_{\Omega} 
  +\delta^{-1} \|\nabla v\|_{U_\delta(\partial \Omega_D)} 
( \delta \| w \|^2_{\partial \Omega_D} 
+ \delta^2 \| \nabla w \|^2_{U_\delta(\partial \Omega_D)} )^{1/2}
\\
&\qquad \qquad 
+ \| \Delta v\|_{U_\delta(\partial \Omega_D)} 
( \delta \| w \|_{\partial \Omega_D} 
+ \delta^2 \| \nabla w \|^2_{U_\delta(\partial \Omega_D)} )^{1/2}
\\
&\qquad \lesssim 
 \|\nabla v\|_\Omega \|\nabla w \|_{\Omega} 
  +\|\nabla v\|_{U_\delta(\partial \Omega_D)} 
( \delta^{-1}\| w \|^2_{\partial \Omega_D} 
+ \| \nabla w \|^2_{U_\delta(\partial \Omega_D)} )^{1/2}
\\
&\qquad \qquad 
+ \delta \| \Delta v\|_{U_\delta(\partial \Omega_D)} 
( \delta^{-1} \| w \|_{\partial \Omega_D} 
+ \| \nabla w \|^2_{U_\delta(\partial \Omega_D)} )^{1/2}
\\
&\qquad \lesssim 
 (\|\nabla v\|^2_\Omega 
     + \|\nabla v\|^2_{U_\delta(\partial \Omega_D)}
     + \delta^2 \| \Delta v\|^2_{U_\delta(\partial \Omega_D)})^{1/2}
\\
&\qquad \qquad \times
(\| \nabla w \|^2_{\Omega} 
+ \delta^{-1} \| w \|^2_{\partial \Omega_D} 
+ \| \nabla w \|^2_{U_\delta(\partial \Omega_D)} )^{1/2}
\\
&\qquad \lesssim 
 (\|\nabla v\|^2_\Omega 
     + \delta^2 \| \Delta v\|^2_{U_\delta(\partial \Omega_D)})^{1/2}
\\
&\qquad \qquad \times
(\| \nabla w \|^2_{\Omega} 
+ \delta^{-1} \| w \|^2_{\partial \Omega_D}  )^{1/2}
\end{align}
Finally, choosing $\delta \sim h$ and using the fact that 
$U_\delta (\partial \Omega_D) \subset \mcT_{h,D}$ we obtain
\begin{equation}
(\nabla v, \nabla w )_{\Omega} 
 - (n\cdot \nabla v, w)_{\partial \Omega_D}
 \lesssim 
 \tn v \tn_h   \tn w \tn_{h,\bigstar} 
\end{equation}
which in combination with (\ref{eq:continuity-b}) concludes the proof.
\end{proof}

\subsection{Interpolation Error Estimates}
\label{section:interpolation}
\paragraph{Definition of the Interpolant.} There is an
extension operator $E:W^k_q(\Omega) \rightarrow W^k_q(\IR^d)$, $k\geq 0$ and $q\geq 1$, 
such that 
\begin{equation}
\| E v \|_{W^k_q(\IR^d)} \lesssim \| v \|_{W^k_q(\Omega)}
\end{equation}
see \cite{Fo95}. Define the interpolant by 
\begin{equation}\label{eq:interpolant}
\pi_{h} : H^s(\Omega) \ni u \mapsto \pi_{Cl,h} ( E u ) \in V_h  
\end{equation} 
where $\pi_{Cl,h}$ is a Clement type interpolation operator 
onto the spline space. We have the expansion
\begin{equation}
 \pi_h ( E v ) = \sum_{\varphi_i \in I} (\pi_h ( E v ))_i \varphi_i
\end{equation}
where $(\pi_h ( E v ))_i$ is the coefficient corresponding to basis function 
$\varphi_i$. We define the interpolant on the active and removed finite element spaces by
\begin{equation}\label{eq:piha}
\pi_{h,a} v = \sum_{\varphi_i \in I_a}  (\pi_h (E v) )_i \varphi_i 
\end{equation}
and 
\begin{equation}\label{eq:pihr}
\pi_{h,r} v = \sum_{\varphi_i \in I_r}  (\pi_h (E v) )_i \varphi_i 
\end{equation}
We then have 
\begin{equation}\label{eq:interpol-split}
\pi_h (E v) = \pi_{h,a} ( E v ) + \pi_{h,r} (E v ) 
\end{equation}
Below we simplify the notation and write $v = Ev$ and 
$\pi_h ( E v ) = \pi_h v$.

\paragraph{Basis Function Removal Condition.} Let $B_r$, 
with corresponding index set $I_r$, be such that
\begin{align}\label{eq:condition}
\sum_{i \in I_r} \tn \varphi_i \tn^2_{h,\bigstar} 
\lesssim tol^2
\end{align}

\paragraph{Selection Procedure.} To determine $B_r$ we may thus compute 
$\tn \varphi_i \tn_{h,\bigstar}$, $i \in I$, sort the basis functions in increasing order and then simply add functions to $I_r$ as long 
as (\ref{eq:condition}) is satisfied. If we wish to avoid computing 
$\tn \varphi_i \tn_{h,\bigstar}$ we may use the directly available diagonal values $A_h(\varphi_i,\varphi_i)$ of the stiffness matrix 
as approximations.

\begin{lem} \label{lem:interpol}
{\bf (Interpolation Error Estimate)} Let $\pi_{h,a}$ be defined by 
(\ref{eq:piha}) with $B = B_a \cup B_r$ such 
that $B_r$ satisfies (\ref{eq:condition}), then 
\begin{align}\label{eq:interpol-errorest-star}
\tn v  - \pi_{h,a} v \tn_{h,\bigstar} 
&\lesssim 
(h^{p} + tol) \| v \|_{H^{p+1}(\Omega)} 
\end{align}
\end{lem}
\begin{proof} Using the identity 
$\pi_h v = \pi_{h,a} v + \pi_{h,r} v$
and the triangle inequality
\begin{align}
\tn v - \pi_{h,a} v \tn^2_{h,\bigstar} 
&\lesssim 
\tn v - \pi_{h} v \tn^2_{h,\bigstar} + \tn \pi_{h,r} v \tn^2_{h,\bigstar}
\\
&\lesssim 
h^{2p} \| v \|^2_{H^{p+1}(\Omega)} + \tn \pi_{h,r} v \tn^2_{h,\bigstar}
\end{align}
by standard spline interpolation results \cite{IGA-h-refined}. To estimate the second term on the right hand side we introduce the scalar product 
\begin{equation}
\langle v,w \rangle_{h,\bigstar} 
= (\nabla v, \nabla w)_\Omega 
+ h (n\cdot \nabla v, n\cdot \nabla w)_{\partial \Omega_D} 
+ h^{-1} (v, w)_{\partial \Omega_D}
+ h^2(\Delta v,\Delta w)_{\mcT_{h,D} \cap \Omega}
\end{equation}
associated with the norm $\tn \cdot \tn_{h,\bigstar}$. Expanding $\pi_{h,r} v$ in the 
basis $B_r$ we get 
\begin{align}
\tn \pi_{h,r} v \tn^2_{h,\bigstar} 
&=\sum_{i,j\in I_r} (\pi_h v)_i (\pi_h v)_j \langle \varphi_i,\varphi_j \rangle_{h,\bigstar}
\\
&\leq \sum_{i \in I_r} \sum_{j\in I_r} \delta_{ij} |(\pi_h v)_i|\, |(\pi_h v)_j| \, \tn \varphi_i \tn_{h,\bigstar}  \tn \varphi_j \tn_{h,\bigstar}  
\\
&\leq \sum_{i \in I_r} \sum_{j\in I_r} \frac{\delta_{ij}}{2} |(\pi_h v)_i|^2  \tn \varphi_i \tn_{h,\bigstar}
+
\frac{\delta_{ij}}{2} |(\pi_h v)_j|^2  \tn \varphi_j \tn_{h,\bigstar}
\\
&= 
\sum_{i \in I_r} \left( \sum_{j\in I_r} \delta_{ij} \right) |(\pi_h v)_i|^2  \tn \varphi_i \tn^2_{h,\bigstar} 
\\
&\lesssim 
 \| \pi_h v \|^2_{L^{\infty}(\mcN_h(\Omega))} 
 \left(  \sum_{i \in I_r}   \tn \varphi_i \tn^2_{h,\bigstar}  \right)
 \\
&\lesssim 
 \| v \|^2_{H^{p+1}(\Omega)} tol^2 
\end{align}
Here 
\begin{itemize}
\item We defined 
\begin{equation}
\delta_{ij} =
\begin{cases}
1 & \text{if $\supp(\varphi_i) \cap \supp(\varphi_j) \neq \emptyset$}
\\
0 & \text{if $\supp(\varphi_i) \cap \supp(\varphi_j) = \emptyset$}
\end{cases}
\end{equation}
and we have the bound 
\begin{equation}
\sum_{j \in I_r} \delta_{ij}
\leq (2p+1)^d
\end{equation}

\item We used the $L^\infty(\mcN_h(\Omega))$ stability of the interpolant 
$\pi_h$ and then the $L^\infty$ stability of the extension 
operator and finally the Sobolev embedding theorem
\begin{equation}
 \| \pi_h v \|_{L^{\infty}(\mcN_h(\Omega))} 
 \lesssim
  \| v \|_{L^{\infty}(\mcN_h(\Omega))}
  \lesssim
  \| v \|_{L^{\infty}(\Omega)}   
  \lesssim 
  \| v \|_{H^{p+1}(\Omega)}
\end{equation}
\end{itemize}
\end{proof}

\subsection{Error Estimate} 

We have the following error estimate.

\begin{thm} Let $u_{h,a}$ be the solution to (\ref{eq:method-nonsym}) with $V_{h,a} = \Span\{B_a\}$ the active spline space, $V_h = \Span\{B\}$ the full spline space, and $B = B_a \cup B_r$, where $B_r$ satisfies (\ref{eq:condition}) with $tol\sim h^p$ , then 
\begin{equation}\label{eq:error-estimate-energy}
\tn u - u_{h,a} \tn_h \lesssim h^{p} \| u \|_{H^{p+1}(\Omega)} 
\end{equation}
\end{thm}
\begin{proof} 
Using coercivity (\ref{eq:coercivity}), 
Galerkin orthogonality (\ref{eq:galort}), 
and continuity (\ref{eq:continuity}), we 
obtain
\begin{align}
\tn u - u_{h,a} \tn^2_h 
&\lesssim 
A_h( u - u_{h,a}, u - u_{h,a} ) 
\\
&= 
A_h( u - u_{h,a}, u - \pi_{h,a} u ) 
\\
&\lesssim 
\tn u - u_{h,a} \tn_{h} \tn u - \pi_{h,a} u \tn_{h,\bigstar}
\end{align}
Thus we arrive at 
\begin{equation}
\tn u - u_{h,a} \tn_h 
\lesssim 
\tn u - \pi_{h,a} u \tn_{h,\bigstar}
\end{equation}
which together with the interpolation error estimate 
\eqref{eq:interpol-errorest-star} completes the proof of \eqref{eq:error-estimate-energy}.

\end{proof}

\begin{rem} Note that if we take $\tau=0$, i.e. we use the 
method without least squares stabilization in the vicinity of the Dirichlet boundary. We may still derive an error estimate 
as follows
\begin{align}
\| \nabla ( u - u_{h,a}) \|^2_\Omega + \| u - u_h \|^2_{\partial \Omega_D}
&\lesssim A_h(u - u_{h,a}, u - u_{h,a} ) 
\\
&= 
A_h( u - u_{h,a}, u - \pi_{h,a} u ) 
\\
&\lesssim 
\tn u - u_{h,a} \tn_{h} \tn u - \pi_{h,a} u \tn_{h,\bigstar}
\end{align}
Now we  note that 
\begin{align}
\tn u - u_{h,a} \tn^2_{h} 
&=
\| \nabla ( u - u_{h,a}) \|^2_\Omega + \| u - u_h \|^2_{\partial \Omega_D}
+ h^2 \| \Delta( u - u_{h,a} ) \|^2_{\mcT_{h,D}\cap \Omega}
\\
&=\| \nabla ( u - u_{h,a}) \|^2_\Omega + \| u - u_h \|^2_{\partial \Omega_D}
+ h^2 \| f - \Delta u_{h,a}  \|^2_{\mcT_{h,D}\cap \Omega}
\end{align}
and thus we obtain the bound
\begin{align}
\| \nabla ( u - u_{h,a}) \|^2_\Omega + \| u - u_h \|^2_{\partial \Omega_D}
&\lesssim
h^{2p} \| u \|^2_{H^{p+1}(\Omega)}
+
h^2 \| f - \Delta u_{h,a}  \|^2_{\mcT_{h,D}\cap \Omega}
\end{align}
where the second term on the right hand side is a residual term involving the computed solution $u_h$. The resulting bound is thus 
of a priori - a posteriori type. One may estimate the residual term 
on elements in the interior of $\Omega$ but for elements which 
are cut we do not have access to the required inverse estimate. 
\end{rem}

\subsection{Bounds in Terms of the Geometry of the Cut Elements}

\begin{figure}
\centering
\begin{subfigure}[t]{.24\linewidth}
\includegraphics[width=\linewidth]{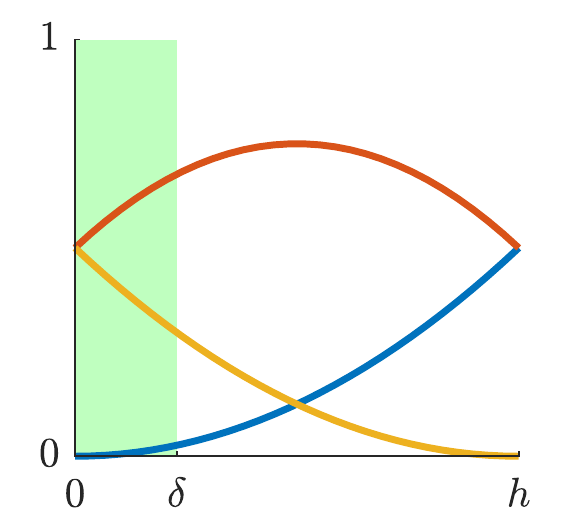}
\caption{$C^1 Q^2$ basis}
\end{subfigure}\
\begin{subfigure}[t]{.24\linewidth}
\includegraphics[width=\linewidth]{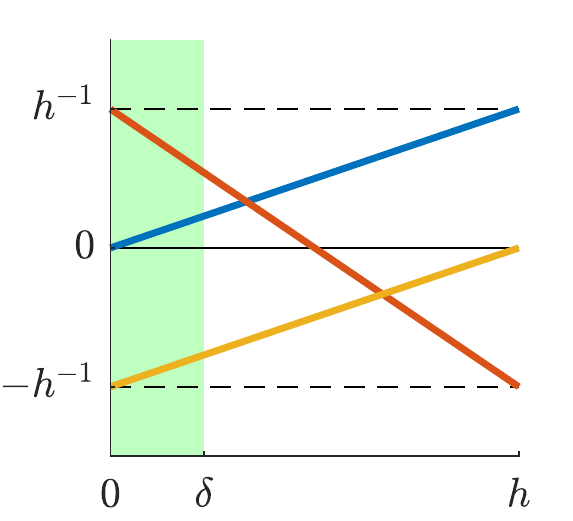}
\caption{$C^1 Q^2$ gradient}
\end{subfigure}\
\begin{subfigure}[t]{.24\linewidth}
\includegraphics[width=\linewidth]{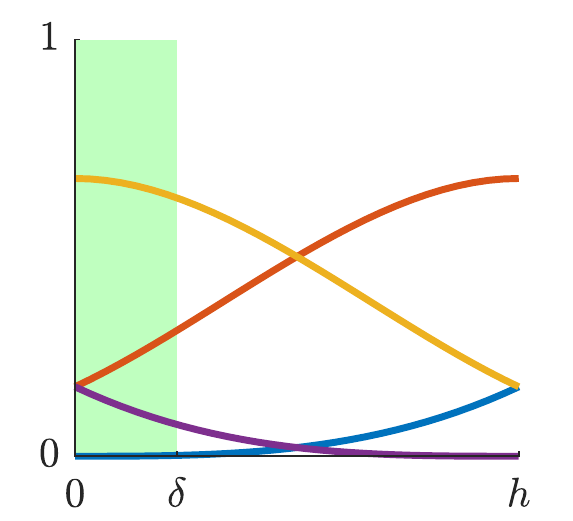}
\caption{$C^2 Q^3$ basis}
\end{subfigure}\
\begin{subfigure}[t]{.24\linewidth}
\includegraphics[width=\linewidth]{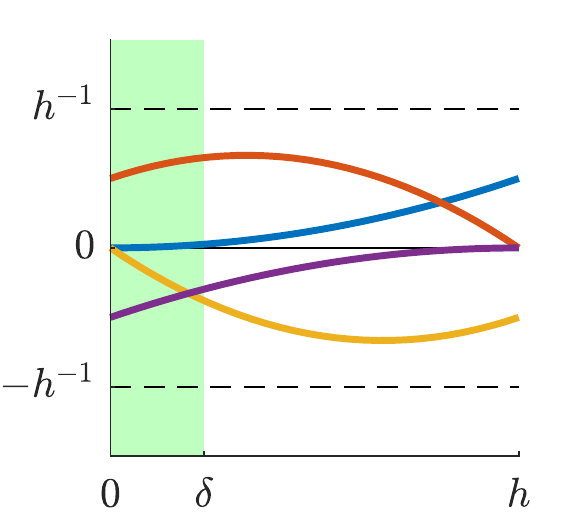}
\caption{$C^2 Q^3$ gradient}
\end{subfigure}

\begin{subfigure}[t]{.24\linewidth}
\includegraphics[width=\linewidth]{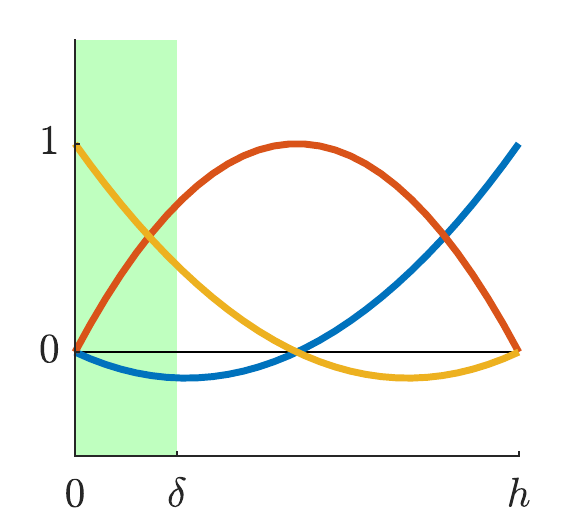}
\caption{$Q^2$ basis}
\end{subfigure}\
\begin{subfigure}[t]{.24\linewidth}
\includegraphics[width=\linewidth]{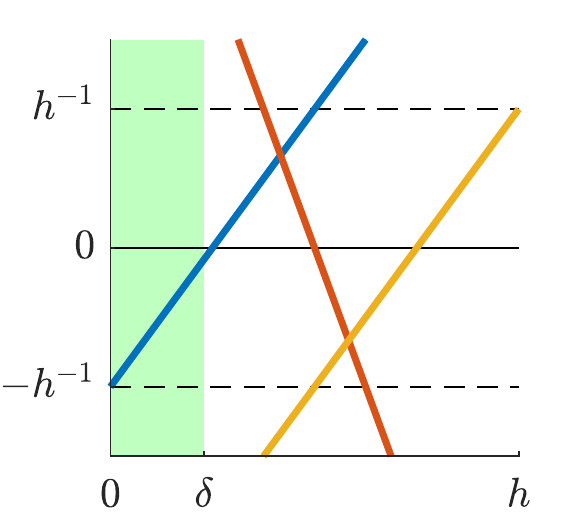}
\caption{$Q^2$ gradient}
\end{subfigure}\
\begin{subfigure}[t]{.24\linewidth}
\includegraphics[width=\linewidth]{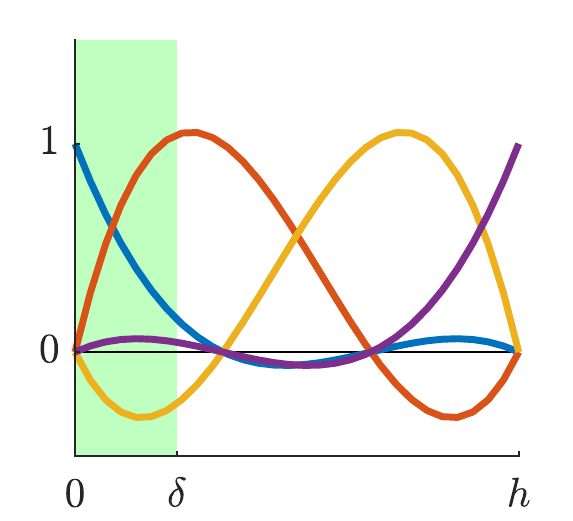}
\caption{$Q^3$ basis}
\end{subfigure}\
\begin{subfigure}[t]{.24\linewidth}
\includegraphics[width=\linewidth]{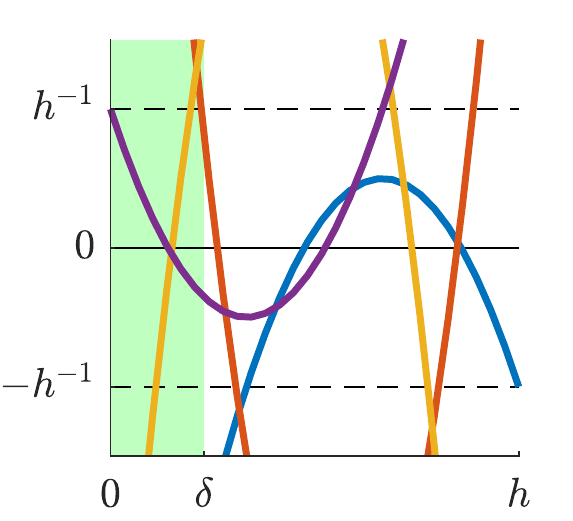}
\caption{$Q^3$ gradient}
\end{subfigure}

\caption{B-spline (top row) and Lagrange (bottom row) basis functions of order $p=2,3$ in a $1D$ element intersecting $\Omega$. Note that gradient of the blue B-spline basis functions is $O(h^{-1}(\frac{\delta}{h})^{p-1})$ within $\Omega$ while the gradient of Lagrange basis functions is $O(h^{-1})$ regardless of $p$.}
\label{fig:spline-vs-lagrange}
\end{figure}

In this section we derive a criterion in terms of the 
geometry of the cut support of the basis function which 
implies (\ref{eq:condition}). 
This criterion will in general not be used in practice 
but it provides insight into the effect of the higher order 
regularity of the B-splines. 

Assuming that there are $h^{-(d-1)}$ such elements we 
have the estimate 
\begin{equation}
\sum_{i \in I_r} \tn \varphi_i \tn^2_{h,\bigstar} 
\lesssim h^{-(d-1)} \max_{i \in I_r } 
\tn \varphi_i \tn^2_{h,\bigstar} 
\end{equation}
and setting $tol\sim h^p$ we get 
\begin{equation}
\max_{i \in I_r } 
\tn \varphi_i \tn^2_{h,\bigstar} 
\lesssim 
h^{d-1} tol \lesssim h^{2p + d -1} 
\end{equation}
and we may define $B_r$ as all basis functions $\varphi \in B$ 
such that 
\begin{equation}
\tn \varphi \tn^2_{h,\bigstar} 
\lesssim 
h^{d-1} tol \lesssim h^{2p + d -1} 
\end{equation}

Let us for simplicity consider a basis function $\varphi$ 
such that $\supp(\varphi) \subset \partial \Omega_D = 
\emptyset$, i.e. a basis function that reside on the Neumann 
part of the boundary. In this case 
$
\tn \varphi \tn_{h,\bigstar} 
= \| \nabla \varphi \|_{\supp(\varphi) \cap \Omega}
$
and thus $\varphi \in B_r$ if 
\begin{equation}\label{eq:condition-grad}
\| \nabla \varphi \|^2_{\supp(\varphi) \cap \Omega} 
\lesssim h^{2p + d -1} 
\end{equation}

\paragraph{The 1D Case: Energy Norm.} Let $\Omega = [0,1]$ 
and consider a basis function $\varphi$ with support $[X_0, X_1]$ 
such that 
$X_0 \in [0,1]$ and $\supp(\varphi) \cap [0,1] = [X_0,1]$ 
is an interval of length $\delta$. Then for $\delta$ small 
enough we have 
\begin{equation}
\varphi(x) = \left( \frac{x}{h} \right)^p, 
\qquad
|D \varphi (x)|^2 =  \frac{p^2}{h^2}\left( \frac{x}{h} \right)^{2(p-1)}
\end{equation}
up to constants and in local coordinates with origo $X_0$, 
and 
\begin{equation}
\| D \varphi \|^2 
= \int_0^\delta 
\frac{p^2}{h^2}\left( \frac{x}{h} \right)^{2(p-1)}
= \frac{p}{h} \frac{p}{2p-1}\left(\frac{\delta}{h}\right)^{2p - 1}
\end{equation}
Condition (\ref{eq:condition-grad}) thus takes the form
\begin{equation}\label{eq:condition-1D-B-splines}
\frac{p}{h} \frac{p}{2p-1} 
\left(\frac{\delta}{h}\right)^{2p - 1} 
\lesssim h^{2p + d - 1} \quad  \Longrightarrow \quad
\frac{\delta}{h} 
\lesssim 
h^{\frac{2p + 1}{2p - 1}}
\end{equation}
For Lagrange basis functions we instead have $|D\varphi(x)| 
\sim h^{-1}$ and we therefore obtain the condition
\begin{equation}\label{eq:condition-1D-Lagrange}
\delta h^{-2} \lesssim h^{2p + d - 1} 
\quad  \Longrightarrow \quad
\frac{x}{\delta} 
\lesssim 
h^{2p + 1}
\end{equation}
An illustration of both B-spline and Lagrange basis functions in this setting is given in Figure~\ref{fig:spline-vs-lagrange}.
Comparing (\ref{eq:condition-1D-B-splines}) and 
(\ref{eq:condition-1D-Lagrange}) we note that the condition 
is much stronger for the Lagrange functions and higher 
order $p$.

\paragraph{The 1D Case: Max Norm.}
The difference between the B-splines and Lagrange basis 
functions is even more drastic if we consider instead evaluating 
the max norm of the derivative. Then for B-splines we have 
\begin{equation}
\| D \varphi \|_{L^\infty(\supp(\varphi) \cap \Omega)} 
\lesssim h^{-1} \left( \frac{\delta}{h} \right)^{p-1}
\end{equation}
while for Lagrange basis functions
\begin{equation}
\| D \varphi \|_{L^\infty(\supp(\varphi) \cap \Omega)} \lesssim h^{-1}
\end{equation}
which in the latter case can not be controlled by decreasing $\delta$, see Figure~\ref{fig:spline-vs-lagrange}. 
Thus for Lagrange basis functions we get a pointwise error 
of order $h^{-1}$ if we remove a basis function while for quadratic and higher order B-splines we may retain optimal order local 
accuracy by choosing
\begin{equation}
\frac{\delta}{h} \lesssim h^{\frac{p+1}{p-1}}
\end{equation}

\paragraph{The 2D Case: Energy Norm.} 
We now extend our calculation to the 2D case. The higher dimensional case can be handled using a similar approach. Let $X_0$ be a vertex 
of $\supp(\varphi)$ which reside in the interior of $\Omega$. Let $\{e_i\}_{i=1}^d$ be an orthonormal coordinate system centered at $X_0$ and with basis vectors $e_i$, and coordinates $x_i$, aligned with 
the edges $\{E_i\}_{i=1}^d$ of $\supp(\varphi)$ which originates 
at $X_0$, see Figure~\ref{fig:delta-illustration}.
%
Using the local coordinates in the vicinity of $X_0$ we have 
the expansions 
\begin{equation}\label{eq:expansion}
\varphi(x_1,x_2) = \left( \frac{x_1}{h} \right)^p \left(\frac{x_2}{h}\right)^p 
\end{equation}
\begin{equation}\label{eq:expansion-der}
|\nabla \varphi(x_1,x_2)|^2 
= 
\frac{1}{h^2} \left( \frac{x_1}{h} \right)^{2p-2} \left(\frac{x_2}{h}\right)^{2p} 
+ 
 \frac{1}{h^2} \left( \frac{x_1}{h} \right)^{2p} \left(\frac{x_2}{h}\right)^{2p-2} 
\end{equation}
Let $\delta_i = \| X_i - X_0 \|_{\IR^d}$ be the distance from 
the vertex $X_0$ to the intersection $X_i$ of edge $E_i$ with 
the boundary  $\partial \Omega$. Assume that 
\begin{equation}\label{eq:assumption-box}
\supp(\varphi) \cap \Omega \subset [0,\delta_1] \times [0,\delta_2]
\end{equation}
Integrating over $[0,\delta_1] \times [0,\delta_2]$ we obtain
\begin{align}\label{eq:supp-a}
\int_0^{\delta_1} \int_0^{\delta_2} |\nabla \varphi |^2 
&\lesssim 
\left( \frac{\delta_1}{h} \right)^{2p-1} \left(\frac{\delta_2}{h}\right)^{2p+1} 
+ 
 \left( \frac{\delta_1}{h} \right)^{2p+1} \left(\frac{\delta_2}{h}\right)^{2p-1}  
\end{align}
Condition (\ref{eq:condition-grad}) thus takes the form
\begin{equation}
\left( \frac{\delta_1}{h} \right)^{2p-1} \left(\frac{\delta_2}{h}\right)^{2p+1} 
+ 
 \left( \frac{\delta_1}{h} \right)^{2p+1} \left(\frac{\delta_2}{h}\right)^{2p-1} 
 \lesssim 
 h^{2p + d - 1}
\end{equation}
which implies
\begin{align}\label{eq:condition-delta}
\frac{\delta_1}{h} \lesssim h \left( \frac{\delta_2}{h} \right)^{-\frac{2p-1}{2p+1}} 
\qquad 
\text{and} 
\qquad 
\frac{\delta_2}{h} \lesssim h \left( \frac{\delta_1}{h} \right)^{-\frac{2p-1}{2p+1}} 
\end{align}
See Figure \ref{fig:Rh} for an illustration of this condition.
\begin{figure}
\centering
\includegraphics[width=.5\linewidth]{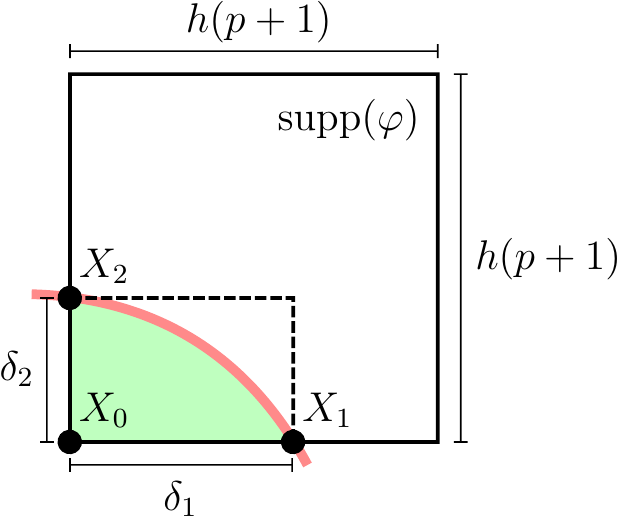}
\caption{Illustration of the geometric quantities used in 
intersection conditions \eqref{eq:condition} in energy norm 
and \eqref{eq:condition-max} in max norm.}
\label{fig:delta-illustration}
\end{figure}

\begin{figure}
\centering
\begin{subfigure}[t]{.3\linewidth}
\includegraphics[width=\linewidth]{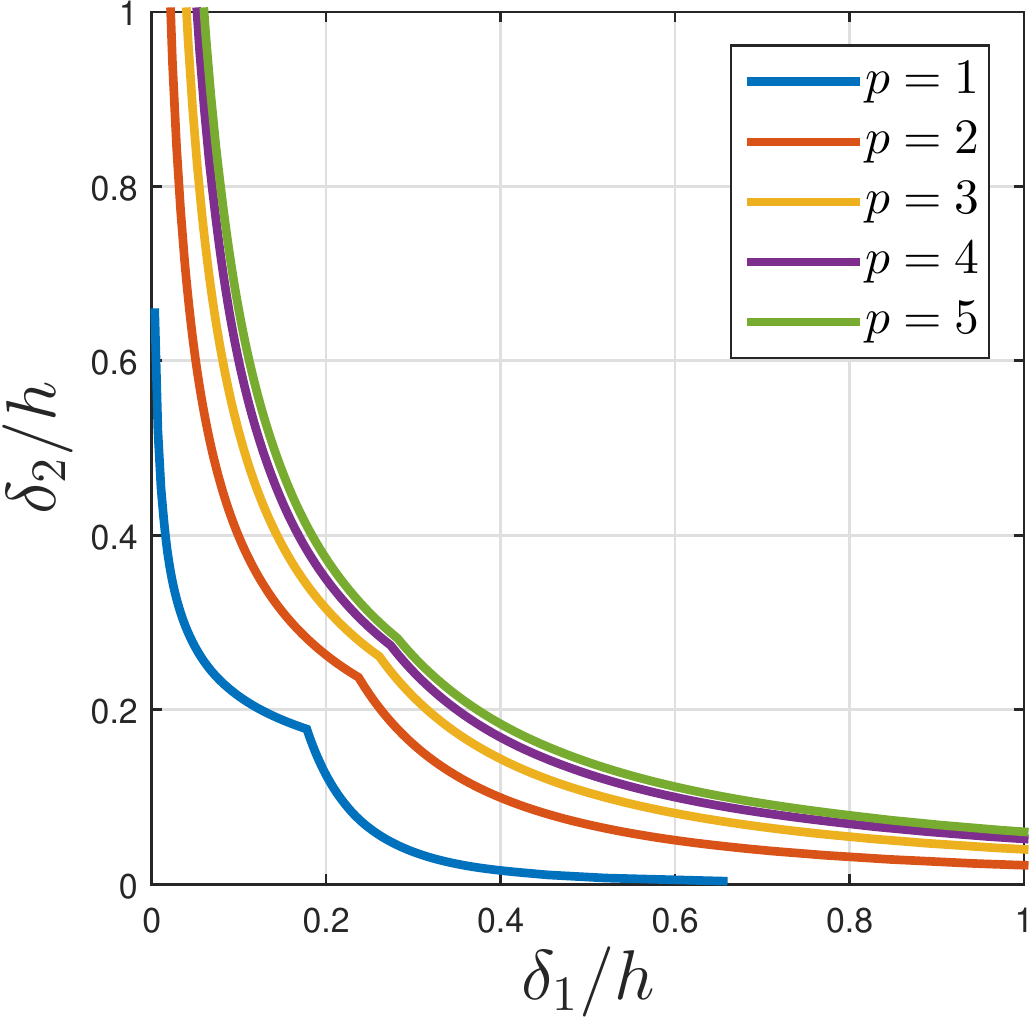}
\caption{Energy norm, $h=0.1$}
\end{subfigure}\quad
\begin{subfigure}[t]{.3\linewidth}
\includegraphics[width=\linewidth]{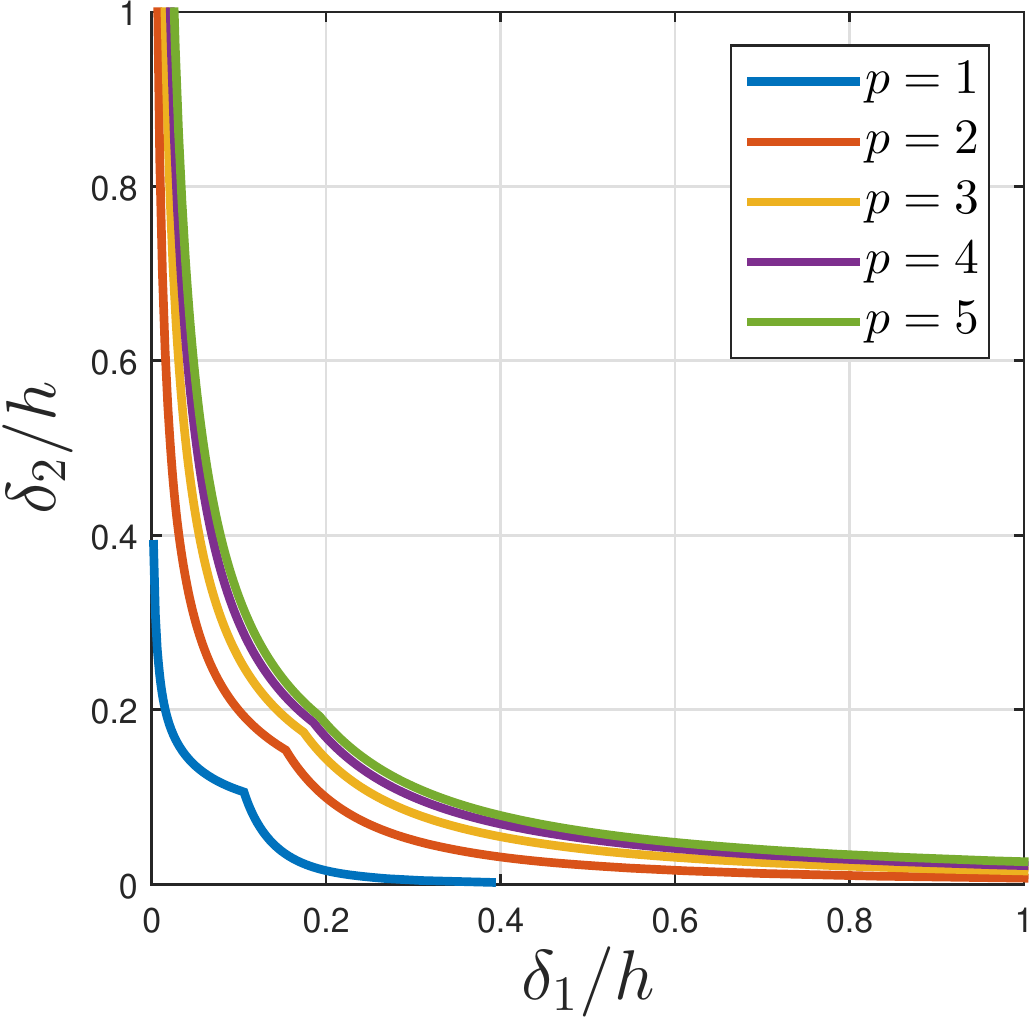}
\caption{Energy norm, $h=0.05$}
\end{subfigure}

\vspace{2ex}
\begin{subfigure}[t]{.3\linewidth}
\includegraphics[width=\linewidth]{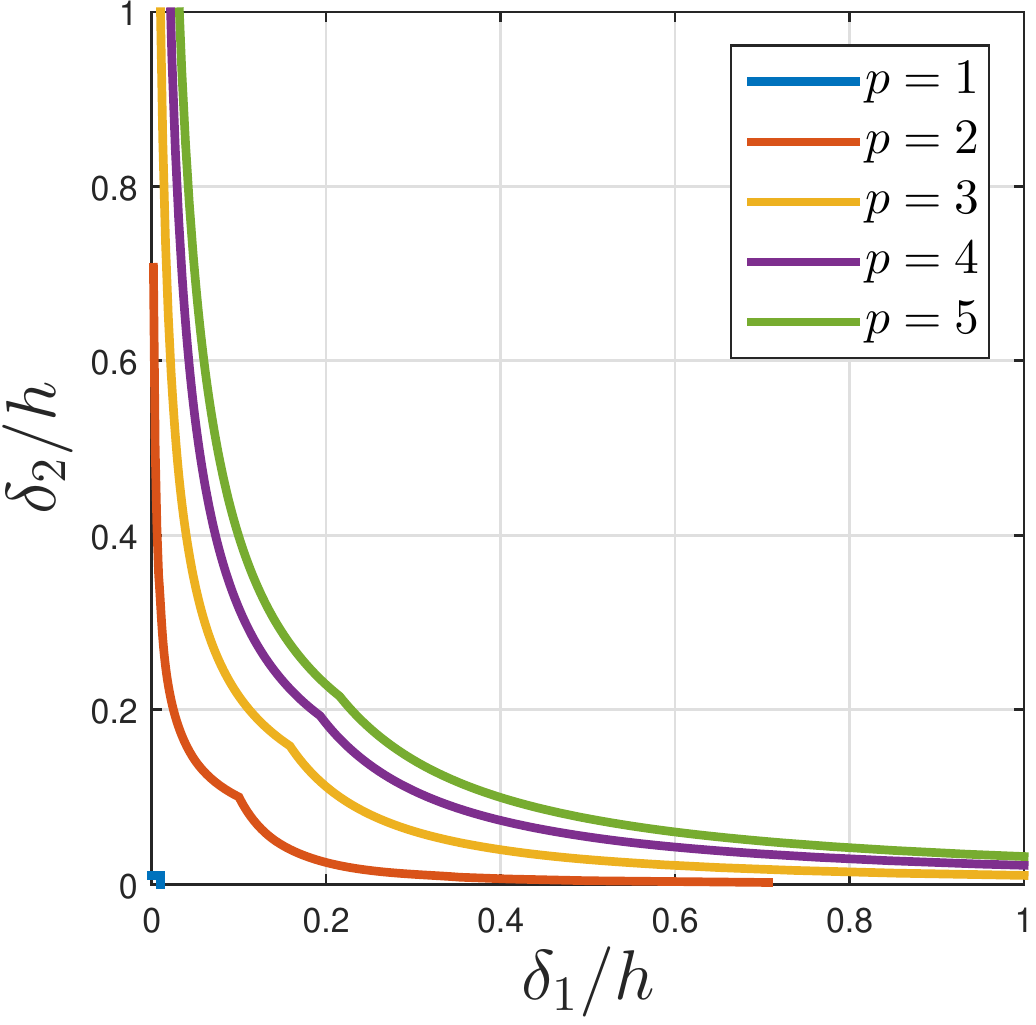}
\caption{Max norm, $h=0.1$}
\end{subfigure}\quad
\begin{subfigure}[t]{.3\linewidth}
\includegraphics[width=\linewidth]{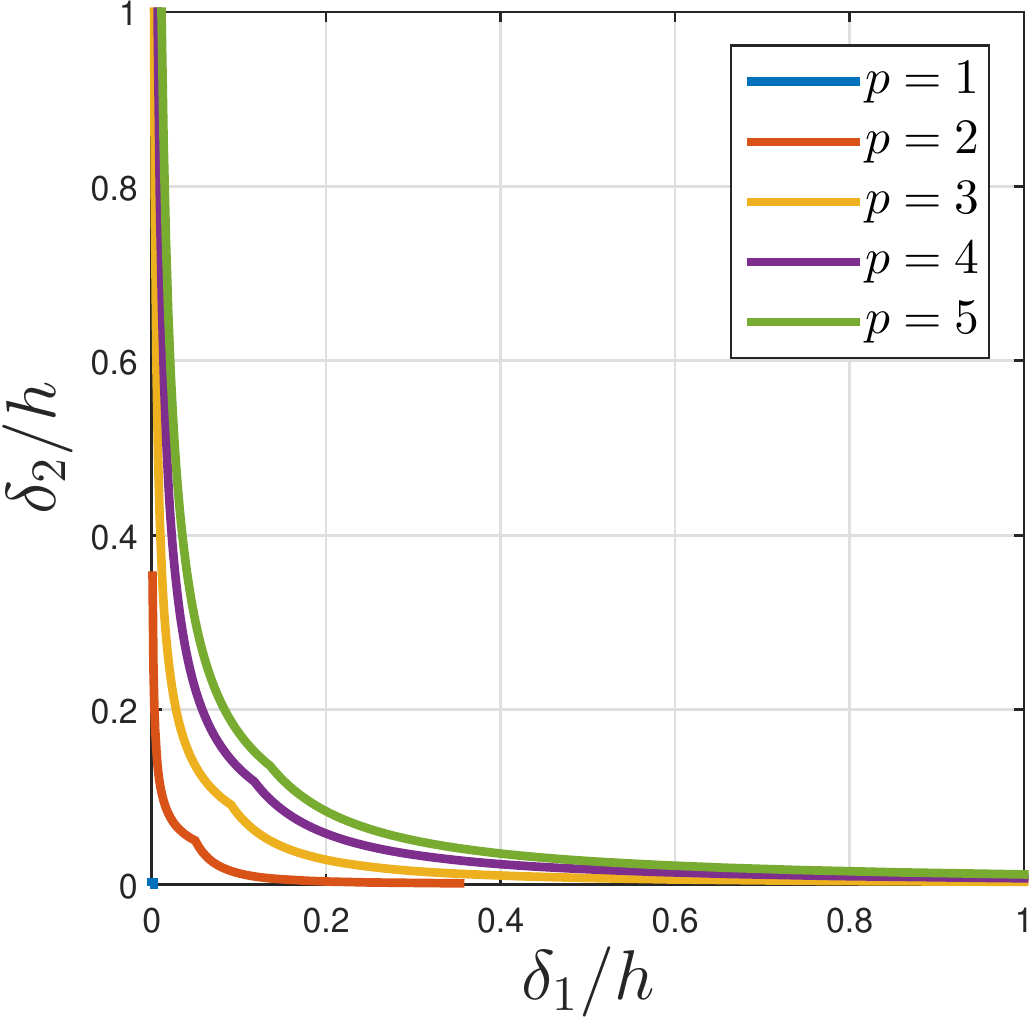}
\caption{Max norm, $h=0.05$}
\end{subfigure}
\caption{Illustrations of the basis function intersection condition \eqref{eq:condition} in energy norm and \eqref{eq:condition-max} in max norm for splines of polynomial order $p=1,2,\dots,5$.}
\label{fig:Rh}
\end{figure}

\paragraph{The 2D Case: Max Norm.} 
Starting from the expansion (\ref{eq:expansion-der}) and 
observing that for small enough $\delta$ parameters $|\nabla \varphi|^2$ is increasing when we move out from the vertex. 
Using assumption (\ref{eq:assumption-box}) we thus conclude 
that 
\begin{equation}
\| \nabla \varphi \|_{L^\infty( \supp (\varphi) \cap \Omega}
\lesssim |\nabla \varphi (\delta_1,\delta_2) |
\end{equation}
We have     
\begin{equation}
\nabla\varphi(x_1,x_2) 
= 
\left[ 
\frac{p}{h}
\left( \frac{x_1}{h} \right)^{p-1} \left(\frac{x_2}{h}\right)^p,
\;
\frac{p}{h}
\left( \frac{x_1}{h} \right)^{p} \left(\frac{x_2}{h}\right)^{p-1}
\right]
\end{equation}
Setting $x_1 = \delta_1$ and $x_2 = \delta_2$ we get the conditions
\begin{equation}
\frac{p}{h} \left( \frac{\delta_1}{h} \right)^{p-1} \left(\frac{\delta_2}{h}\right)^p \lesssim h^p 
\quad \text{and} \quad
\frac{p}{h}
\left( \frac{\delta_1}{h} \right)^{p} \left(\frac{\delta_2}{h}\right)^{p-1} \lesssim h^p
\end{equation}
which we may write in the form
\begin{equation}\label{eq:condition-max}
\frac{\delta_1}{h} \lesssim \frac{1}{p} h^{\frac{p+1}{p}}
\left( \frac{\delta_2}{h} \right)^{-\frac{p-1}{p}}
\quad 
\text{and}
\quad
\frac{\delta_2}{h}\lesssim \frac{1}{p} h^{\frac{p+1}{p}}
\left( \frac{\delta_1}{h} \right)^{-\frac{p-1}{p}}
\end{equation}
See Figure \ref{fig:Rh} for an illustration of this condition.

\section{Numerical Results}

\subsection{Linear Elasticity} \label{section:elasticity}
While we for simplicity use the Poisson model problem in the above analysis the same analysis holds also for other second order elliptic problems which may be of more practical interest. We therefore in the numerical results apply our findings to the  linear elasticity problem: find the displacement
$u:\Omega \rightarrow \IR^d$ such that
\begin{alignat}{2}\label{eq:standard-a}
-\sigma(u)\cdot \nabla &= f \qquad 
&& \text{in $\Omega$}
\\ \label{eq:standard-b}
\sigma(u) \cdot n &= g_N \qquad && 
\text{on $\partial\Omega_N$}
\\ \label{eq:standard-c}
u &= g_D \qquad && 
\text{on $\partial\Omega_D$}
\end{alignat}
where the stress and strain tensors are defined by 
\begin{equation}
\sigma(u) = 2\mu \epsilon(u) + \lambda\text{tr}(\epsilon(u)),
\qquad 
\epsilon(u) = \frac{1}{2}\Big( u\otimes\nabla + \nabla \otimes u \Big)
\end{equation}
with Lam\'e parameters $\lambda$ and $\mu$;
$f$, $g_N$, $g_D$ are given data; $a \otimes b$ 
is the tensor product of vectors $a$ and $b$ with elements
$(a \otimes b)_{ij} = a_i b_j$.

\paragraph{The Nonsymmetric Method for Linear Elasticty.} Find 
$u_{h,a} \in [V_{h,a}]^d$ such that 
\begin{equation}\label{eq:elasticity-method-nonsym}
A_{h}(u_{h,a},v) = L_h(v)\qquad v \in [V_{h,a}]^d
\end{equation}
The forms are defined by
\begin{align}
A_h(v,w) &= a_h(v,w) + \tau h^2(\epsilon(v)\cdot\nabla,\epsilon(w)\cdot\nabla)_{\mcT_{h,D} \cap \Omega}
\\
L_h(v) &= l_h(v) + \tau h^2 (f,\epsilon(v)\cdot\nabla)_{\mcT_{h,D}\cap \Omega}
\end{align}
where 
\begin{align}
a_h(v,w) &= (\sigma(v),\epsilon(w))_\Omega 
- (\sigma(v) \cdot n ,w)_{\partial \Omega_D} 
+ (v, \sigma(w)\cdot n )_{\partial \Omega_D} 
+ \beta h^{-1} (v,w)_{\partial \Omega_D}
\\
l_h(v)&= (f,v)_\Omega 
+ (g_N,v)_{\partial \Omega_N}
+ (g_D,\sigma(v)\cdot n)_{\partial\Omega_D} 
+ \beta h^{-1}(g_D,v)_{\partial \Omega_D}  
\end{align}
with positive parameters $\beta$ and $\tau$. Furthermore, the 
energy norm is defined
\begin{equation}
\tn v \tn_h^2 = ( \sigma(v) , \epsilon(v) )_\Omega
+ h^{-1} \| v \|_{\partial\Omega_D}^2
+ \tau h^2 \| \epsilon(v)\cdot \nabla \|^2_{\mcT_h(\partial \Omega_D) \cap \Omega}
\end{equation}

\begin{figure}
\centering
\begin{subfigure}[t]{.4\linewidth}
\includegraphics[width=\linewidth]{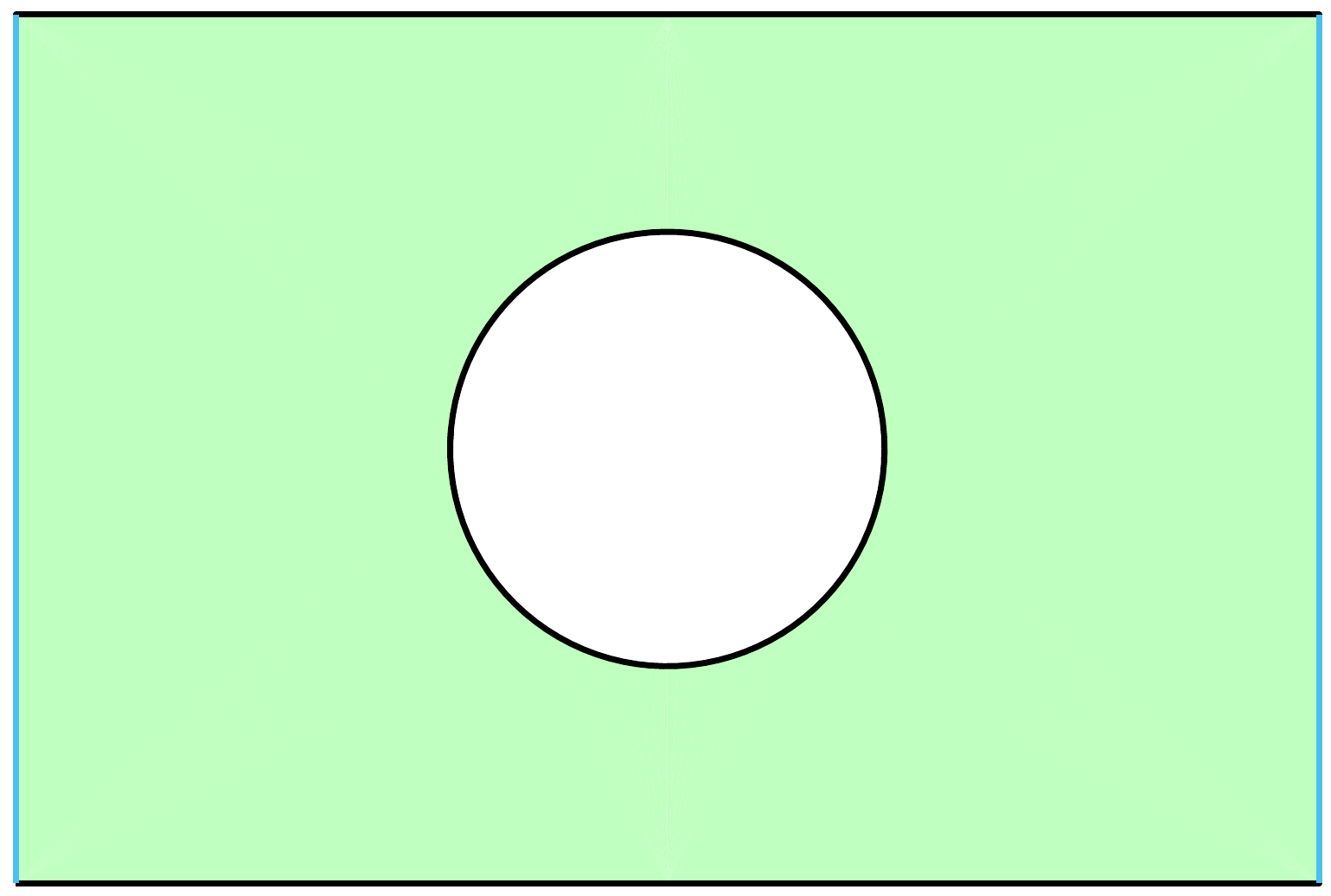}
\caption{Neumann problem}\label{fig:geom-neumann}
\end{subfigure}\qquad
\begin{subfigure}[t]{.3\linewidth}
\includegraphics[width=\linewidth]{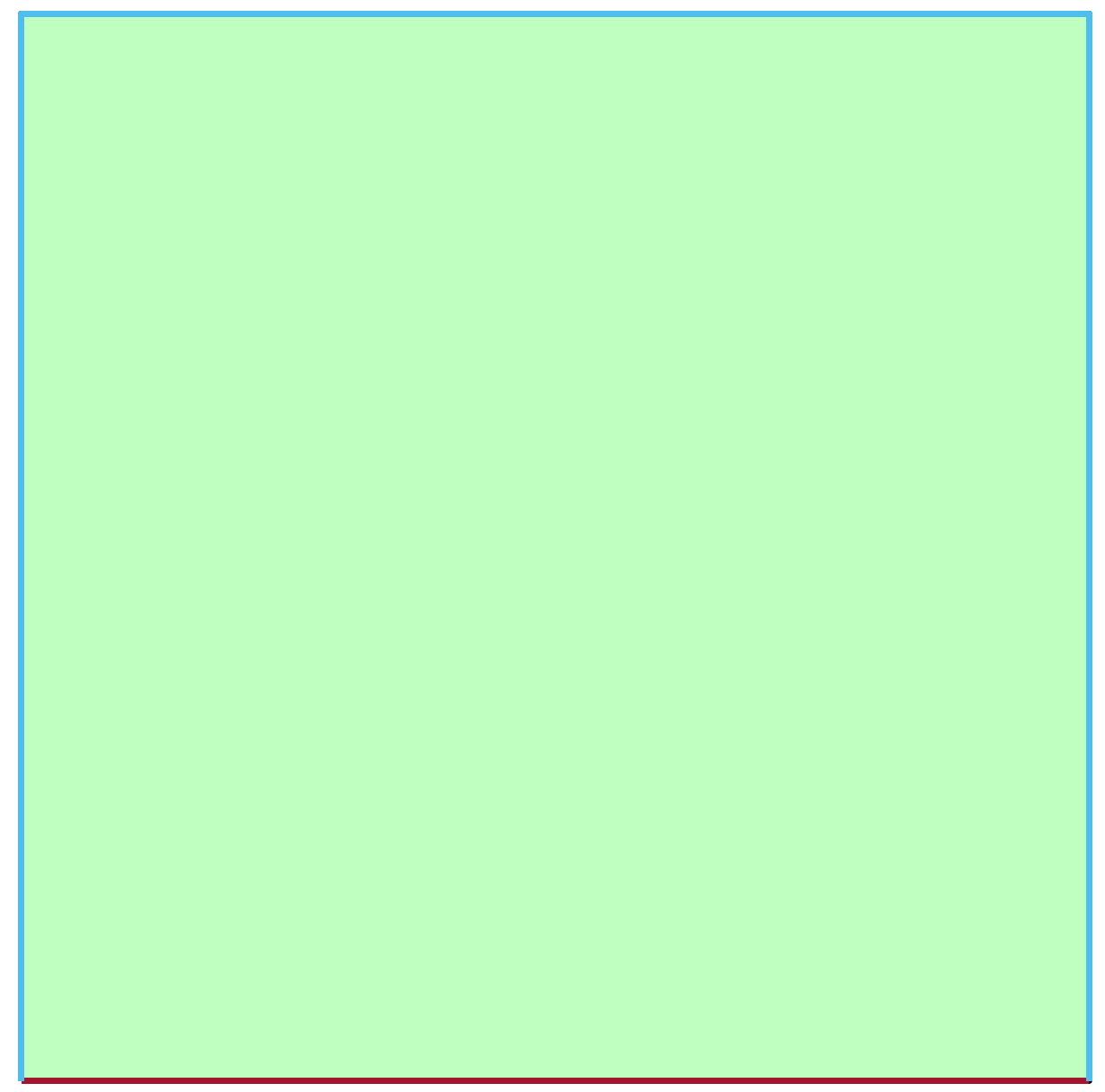}
\caption{Manufactured problem}\label{fig:geom-manufactured}
\end{subfigure}
\caption{Geometries in the two model problems. Boundaries with non-homogeneous Neumann conditions are indicated in blue and Dirichlet boundaries are indicated in red.}
\label{fig:geom}
\end{figure}

\paragraph{A Neumann Problem.}
To illustrate the selection of spline basis functions to remove we first consider a pure Neumann problem with the geometry presented in Figure~\ref{fig:geom-neumann}. The domain is symmetrically pulled from the left and the right using a unitary traction load. We assume a linear isotropic material with an $E$-modulus of $E=100$ and a Poisson ratio of $\nu=0.3$. To ensure the discretized problem is well posed we seek solutions orthogonal to the rigid body modes by using Lagrange multipliers.

\paragraph{A Manufactured Problem.}
To numerically estimate convergence rates we use the following manufactured problem from \cite{HaLaLa18}.
The geometry and the solution is given by
\begin{align}
&\Omega = [0,1]^2
 , \
\partial\Omega_D = \{x\in[0,1],y=0\} , \
\partial\Omega_N = \partial\Omega \backslash \partial\Omega_D
\\
&u(x,y) = [ -\cos(\pi x)\sin(\pi y), \, \sin(\pi x/7)\sin(\pi y/3) ]/10
\end{align}
see Figure~\ref{fig:geom-manufactured}. Assuming a linear isotropic material with the material parameters of steel we deduce expressions for the input data $f$, $g_N$ and $g_D$. Note that while this problem does include a Dirichlet boundary $\partial\Omega_D$ we in our current implementation neglect the least squares term in the vicinity of $\partial\Omega_D$, i.e. we choose $\tau=0$.

\subsection{Illustration of the Selection Procedure}
We utilize the selection procedure based on the stiffness matrix proposed in Section~\ref{section:interpolation}.
Some realizations of this selection are visualized in Figure~\ref{fig:mesh-rem-ex} where we note that the selection becomes more restrictive as the mesh size decreases. This is a natural effect as the selection procedure is developed to ensure optimal approximation properties of the active spline space $V_{h,a}$. We also note that when increasing spline order more basis functions are removed when using the same constant in the tolerance $tol = c h^p$. This can also be seen in Figure~\ref{fig:cumsum} where we investigate how the choice of this constant effects the number of removed basis functions.
In Figure~\ref{fig:neumann-example} we note that the use of basis removal is quite effective and also give better quality stresses along the boundary.

\begin{figure}
\centering
\begin{subfigure}[t]{.4\linewidth}
\includegraphics[width=\linewidth]{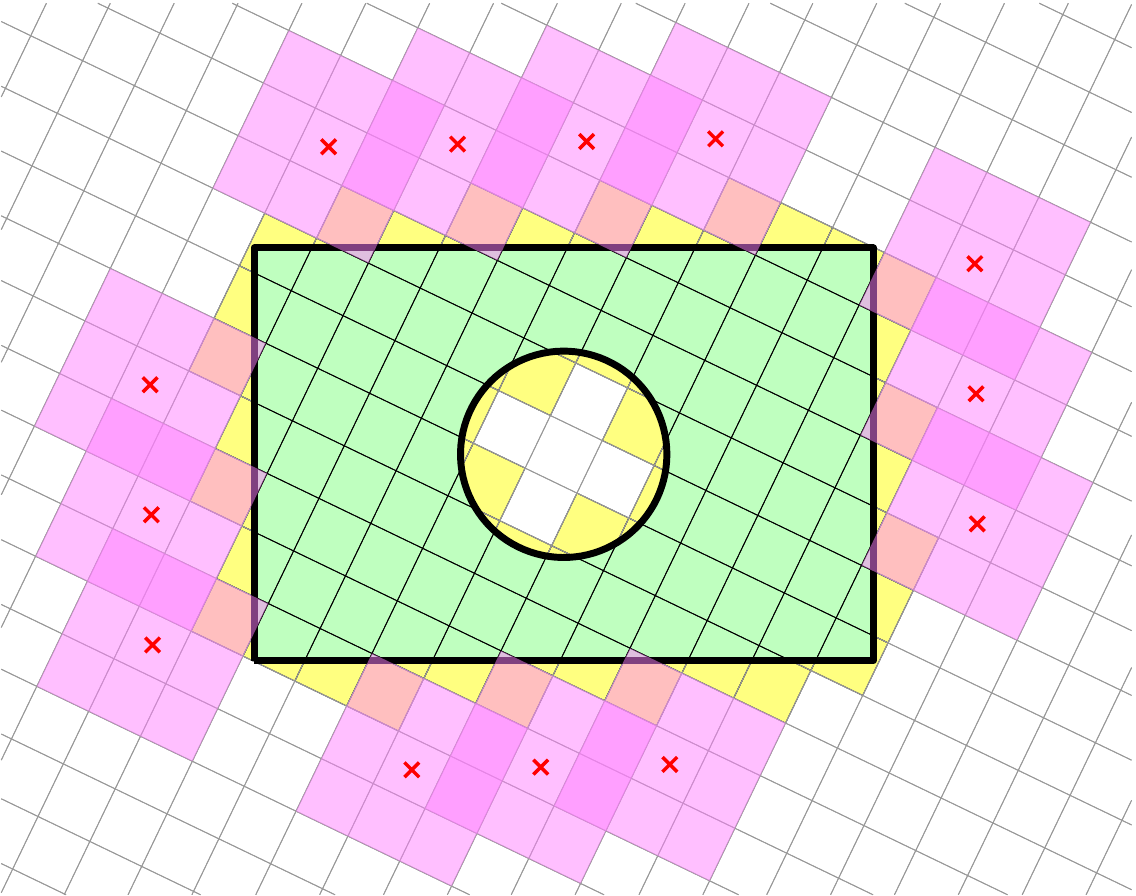}
\caption{$C^1Q^2$, $h=0.4$}
\end{subfigure}\qquad
\begin{subfigure}[t]{.4\linewidth}
\includegraphics[width=\linewidth]{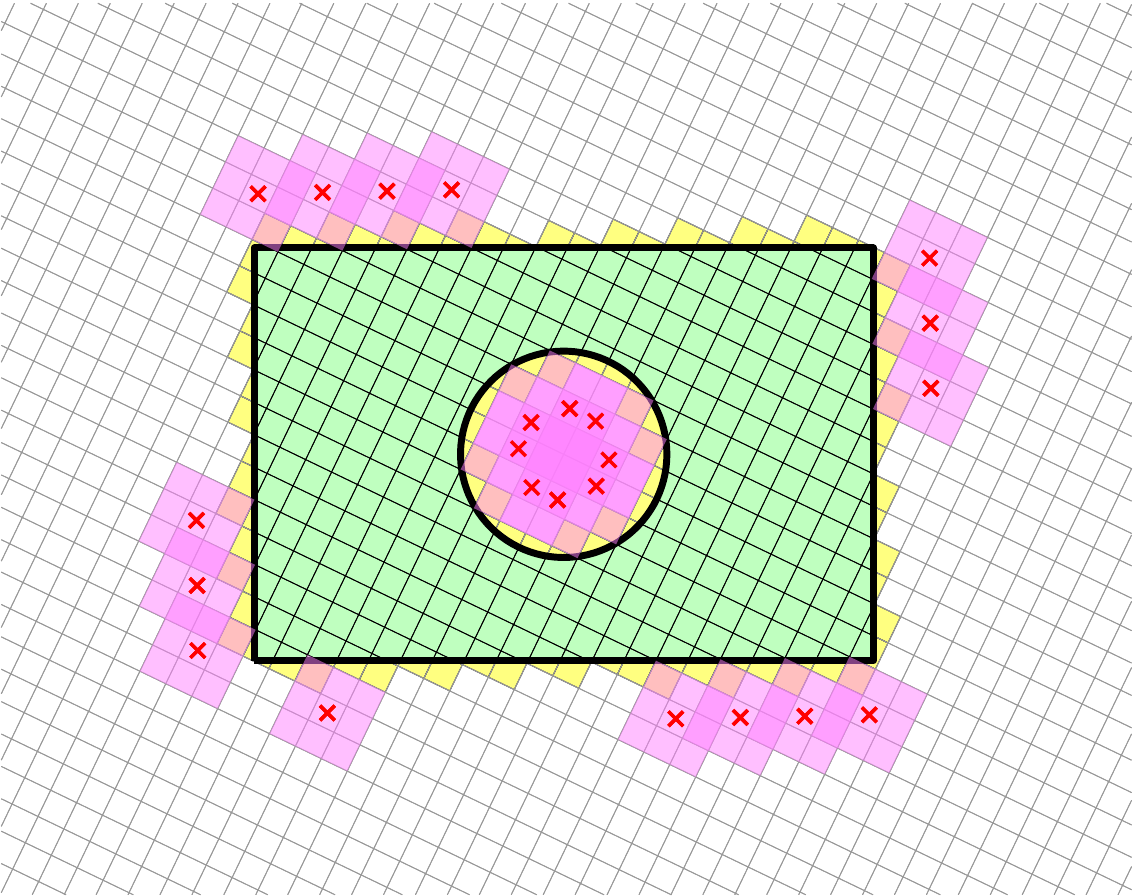}
\caption{$C^1Q^2$, $h=0.2$}
\end{subfigure}

\vspace{2ex}
\begin{subfigure}[t]{.4\linewidth}
\includegraphics[width=\linewidth]{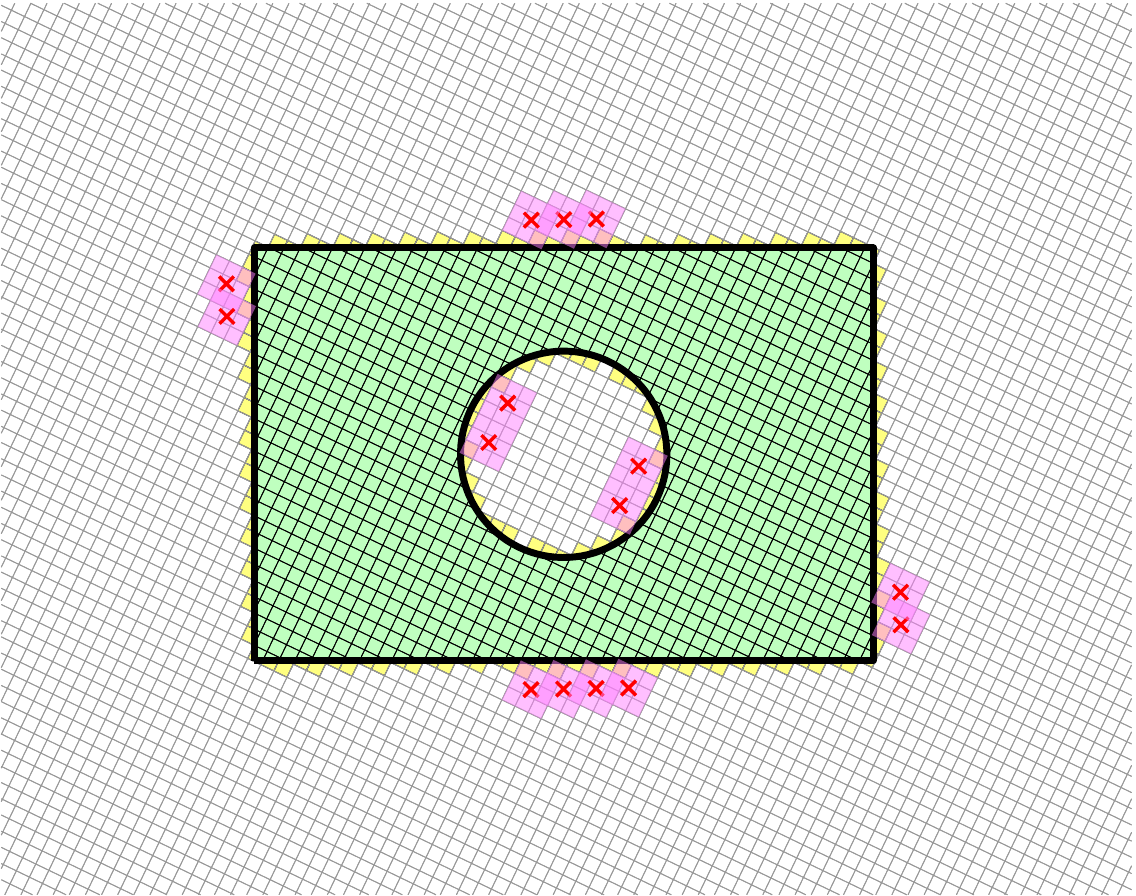}
\caption{$C^1Q^2$, $h=0.1$}
\end{subfigure}\qquad
\begin{subfigure}[t]{.4\linewidth}
\includegraphics[width=\linewidth]{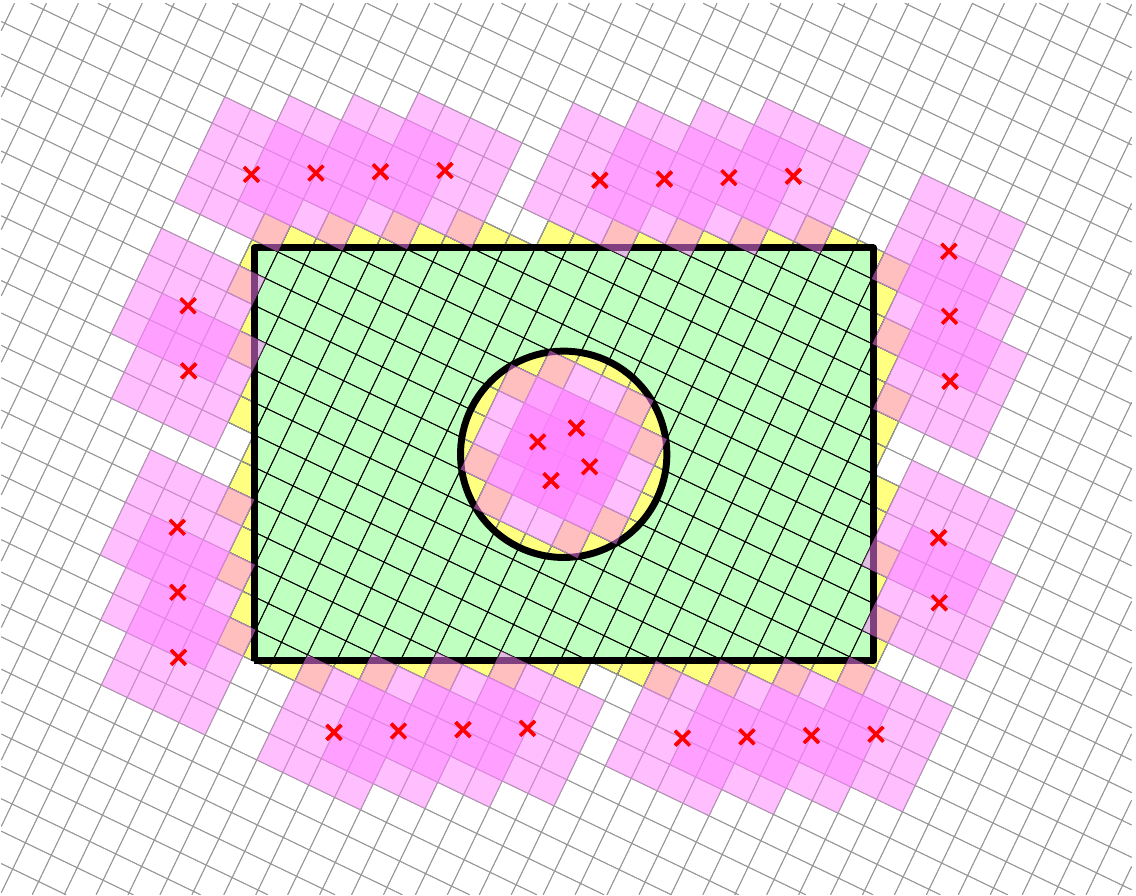}
\caption{$C^2Q^3$, $h=0.2$}
\end{subfigure}
\caption{Four realizations of removed basis functions using on the stiffness matrix based selection procedure described in Section~\ref{section:interpolation}; all using the same constant $c=0.01$ for the tolerance $tol = ch^{p} \times \sqrt{E}$ in \eqref{eq:condition}.
Each cross marks a removed basis function and the domain of its support is visualized in pink.
In (a)--(c) we note that the selection becomes more restrictive with smaller mesh size $h$.
Comparing (b) and (d) we also note that more basis functions typically may be removed as the spline order increases.}
\label{fig:mesh-rem-ex}
\end{figure}

\begin{figure}
\centering
\begin{subfigure}[t]{.35\linewidth}
\includegraphics[width=\linewidth]{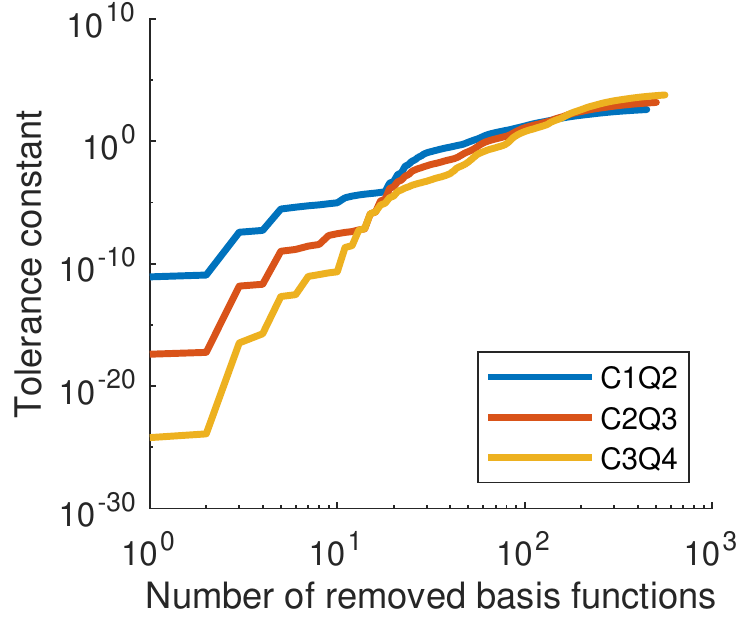}
\caption{$h=0.2$}
\end{subfigure}\qquad
\begin{subfigure}[t]{.35\linewidth}
\includegraphics[width=\linewidth]{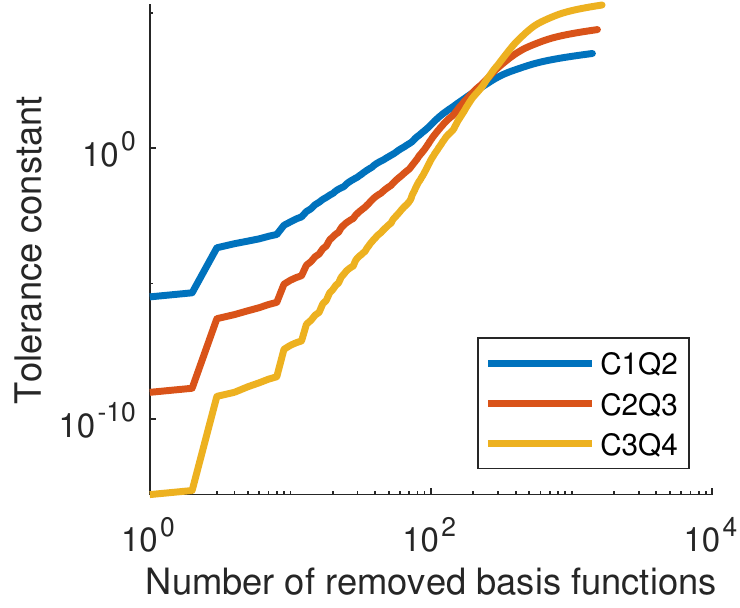}
\caption{$h=0.1$}
\end{subfigure}
\caption{Studies of how the choice of constant $c$ for the tolerance $tol=c h^p \times \sqrt{E}$ in \eqref{eq:condition} relates to the number of removed basis functions. The set-up here is the same as in Figure~\ref{fig:mesh-rem-ex}.}
\label{fig:cumsum}
\end{figure}

\begin{figure}
\centering
\begin{subfigure}[t]{.45\linewidth}
\includegraphics[width=\linewidth]{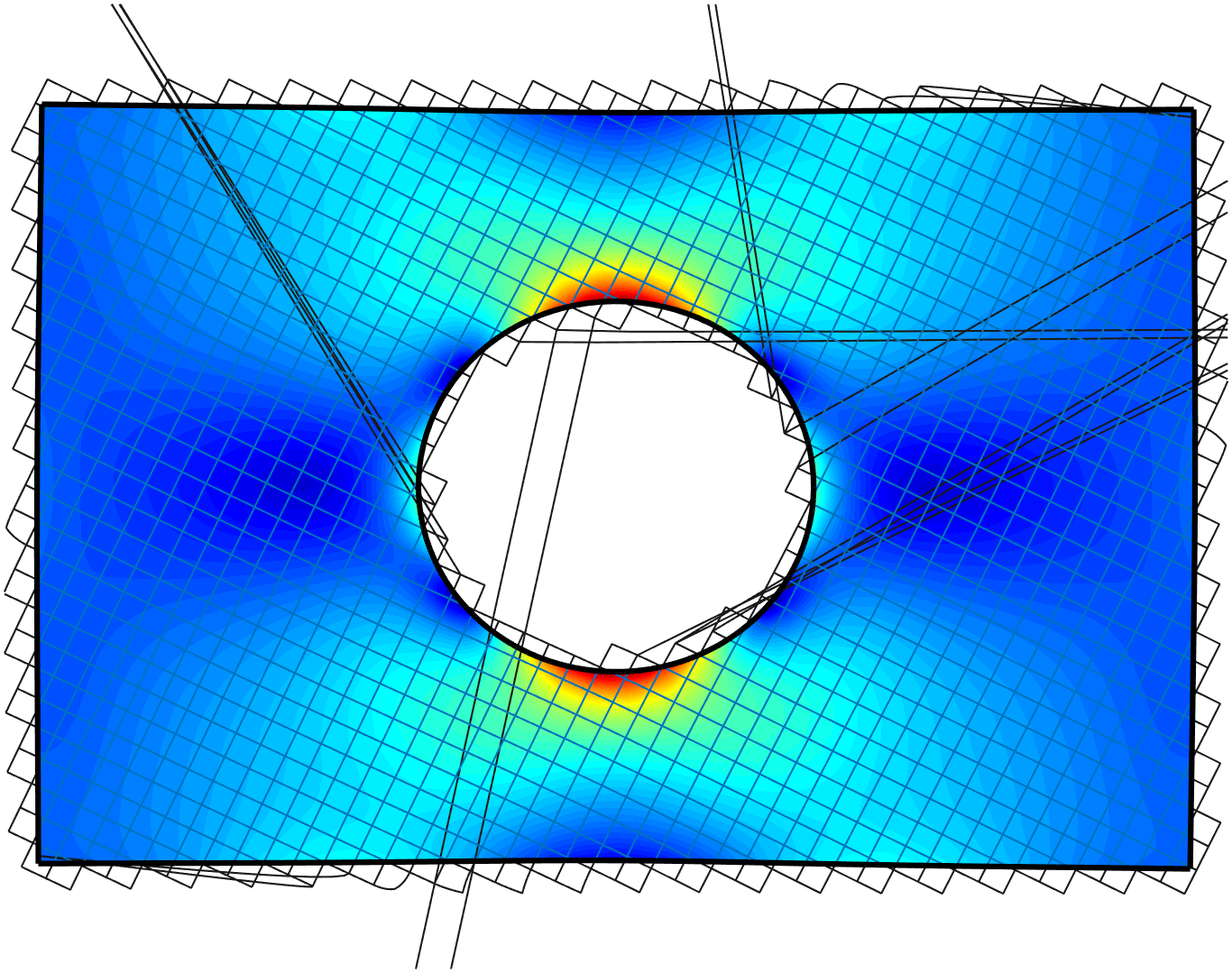}
\caption{Standard solution}
\end{subfigure}\qquad
\begin{subfigure}[t]{.45\linewidth}
\includegraphics[width=\linewidth]{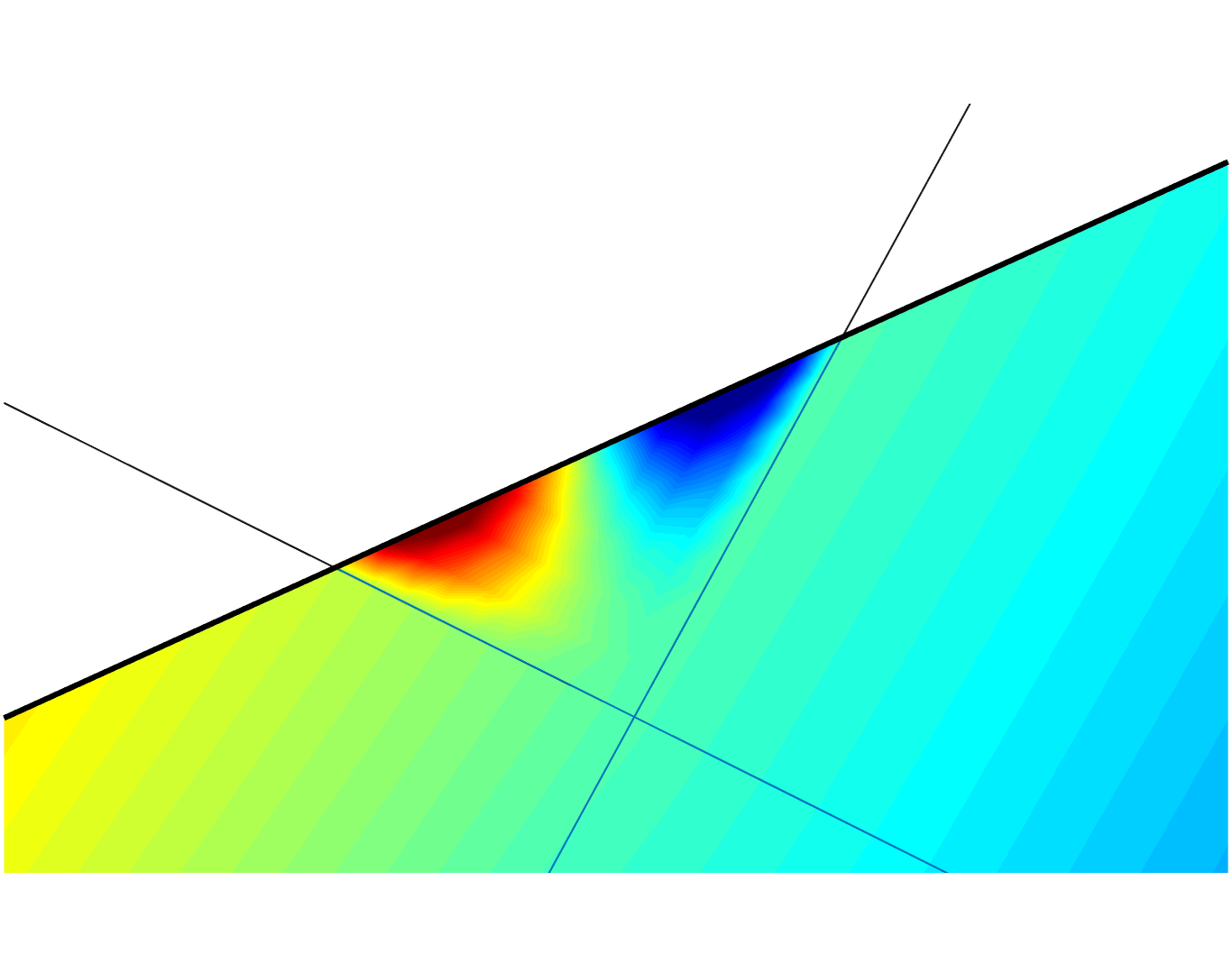}
\caption{Detail in standard solution}
\end{subfigure}

\vspace{2ex}
\begin{subfigure}[t]{.45\linewidth}
\includegraphics[width=\linewidth]{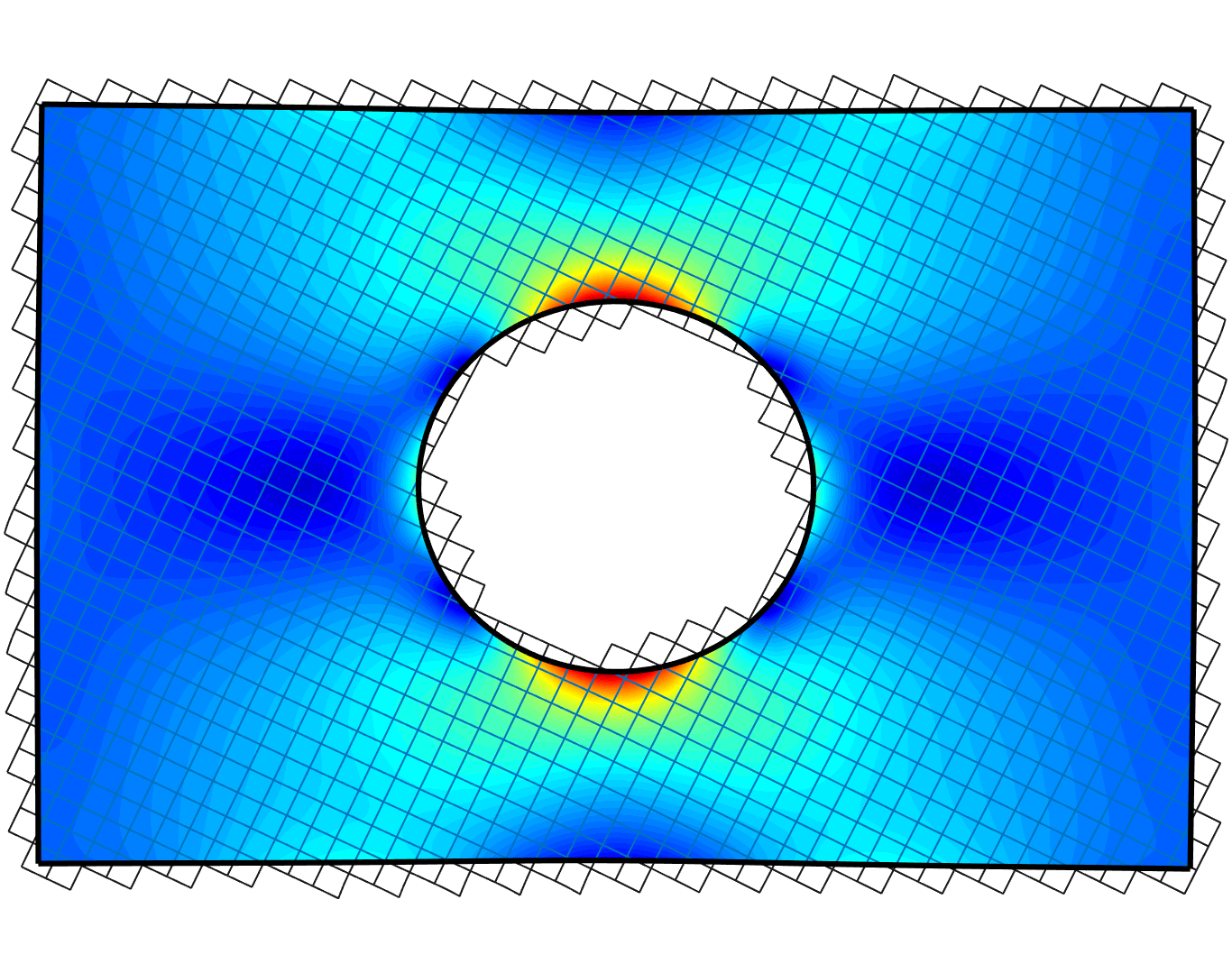}
\caption{Basis removal solution}
\end{subfigure}\qquad
\begin{subfigure}[t]{.45\linewidth}
\includegraphics[width=\linewidth]{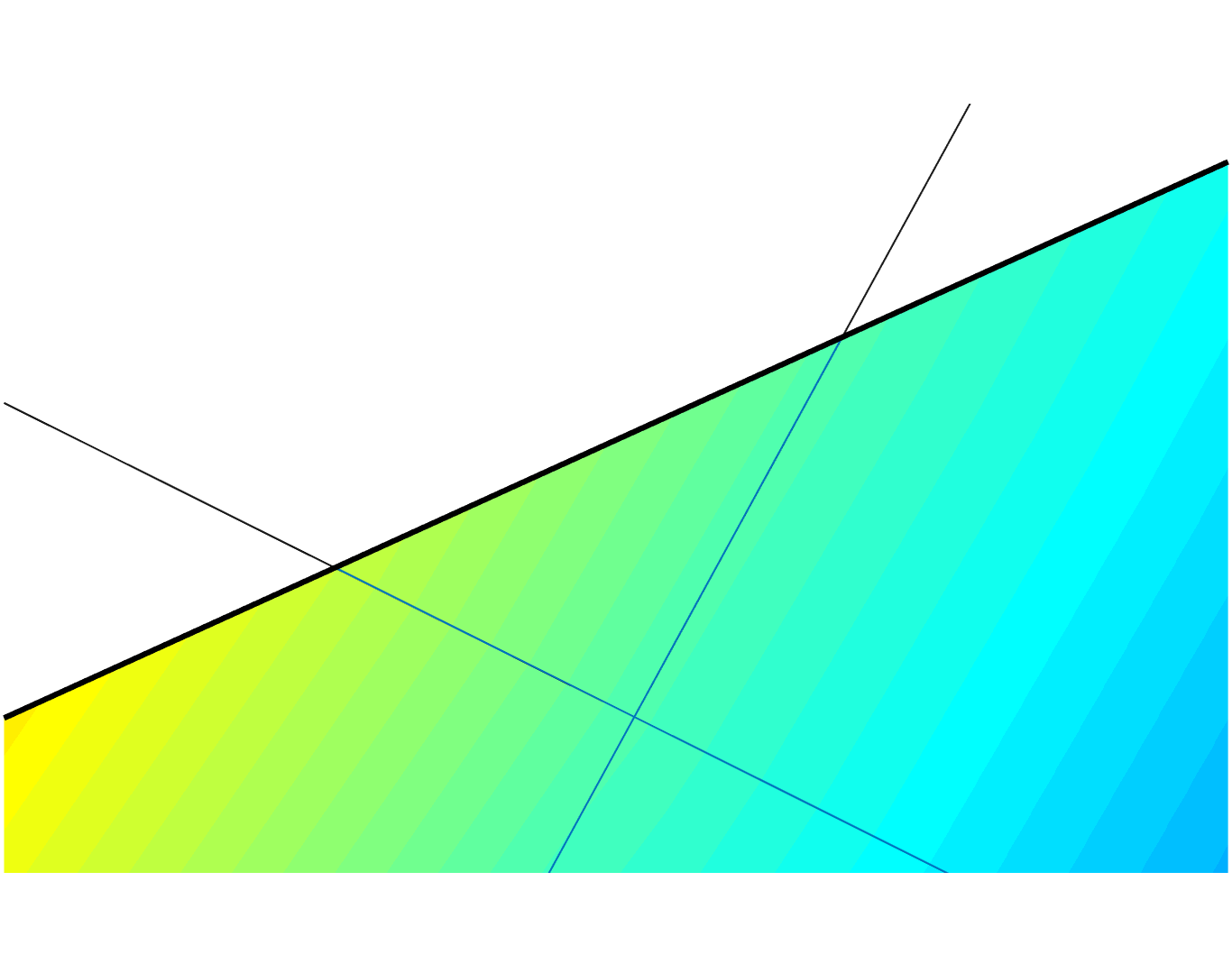}
\caption{Detail in basis removal solution}
\end{subfigure}
\caption{Displacements and von-Mises stresses from numerical solutions with and without basis removal in the Neumann problem using $C^1 Q^2$-splines and mesh size $h=0.1$. In the detailed view we note poor quality of the stresses on the boundary in the standard solution which is remedied when removing the problematic basis function.}
\label{fig:neumann-example}
\end{figure}

\subsection{Convergence}

To estimate the convergence we use the manufactured problem described in Section~\ref{section:elasticity} and the cut situations are induced by rotating the background grid $\pi/7$ radians as illustrated by the mesh with removed basis functions in Figure~\ref{fig:manufactured-example} together with the corresponding numerical solution.
In Figure~\ref{fig:conv-energy-norm} we present convergence studies in energy norm for various choices of the constant $c$ in the tolerance $tol = c h^{p} \times \sqrt{E}$ used in the selection procedure. As can be seen, a larger constant naturally means a larger error, but the convergence rates remain optimal.
The stiffness matrix condition numbers corresponding to these convergence studies
is presented in Figure~\ref{fig:condest}. It can be noted that while basis removal greatly reduce the size of the condition numbers, basis removal alone does not yield an optimal scaling of $O(h^{-2})$.

\begin{figure}
\centering
\begin{subfigure}[t]{.49\linewidth}
\includegraphics[width=\linewidth]{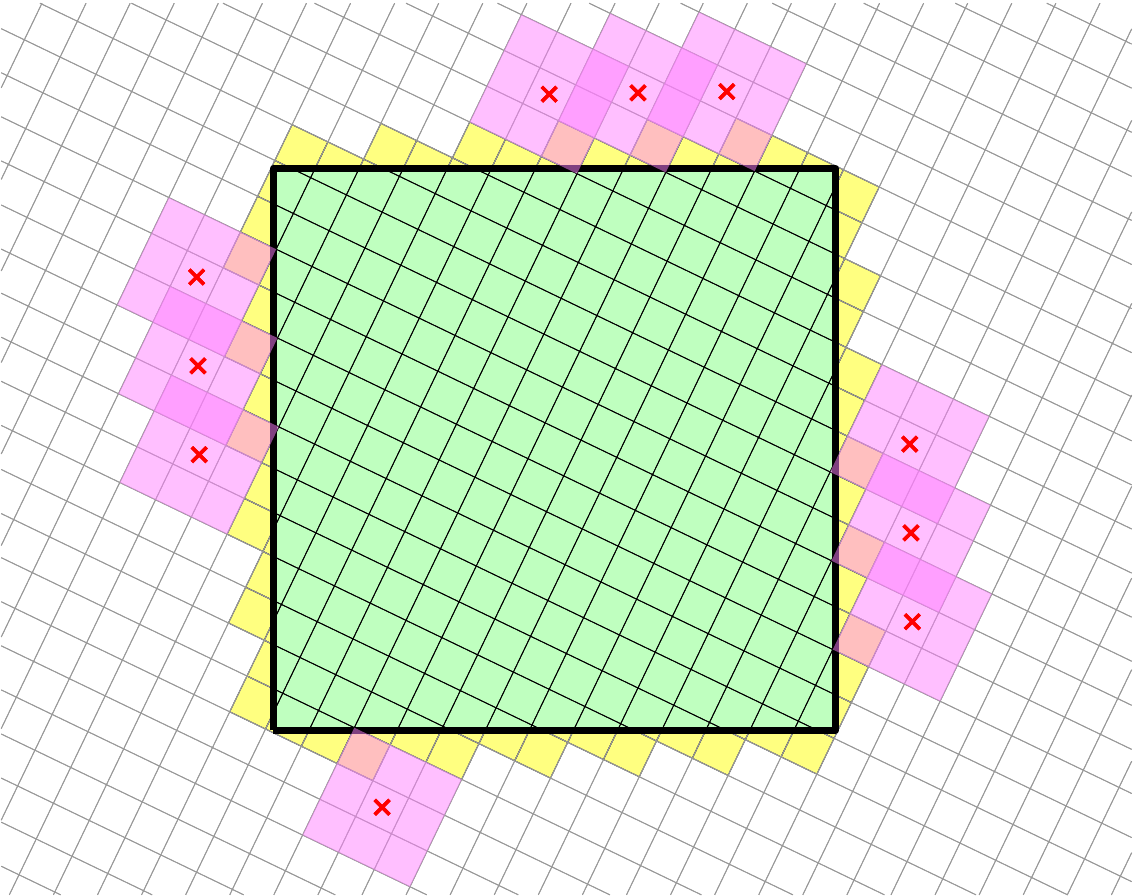}
\caption{Mesh and removed spline basis functions} \label{fig:manufactured-mesh}
\end{subfigure}
\begin{subfigure}[t]{.49\linewidth}
\includegraphics[width=\linewidth]{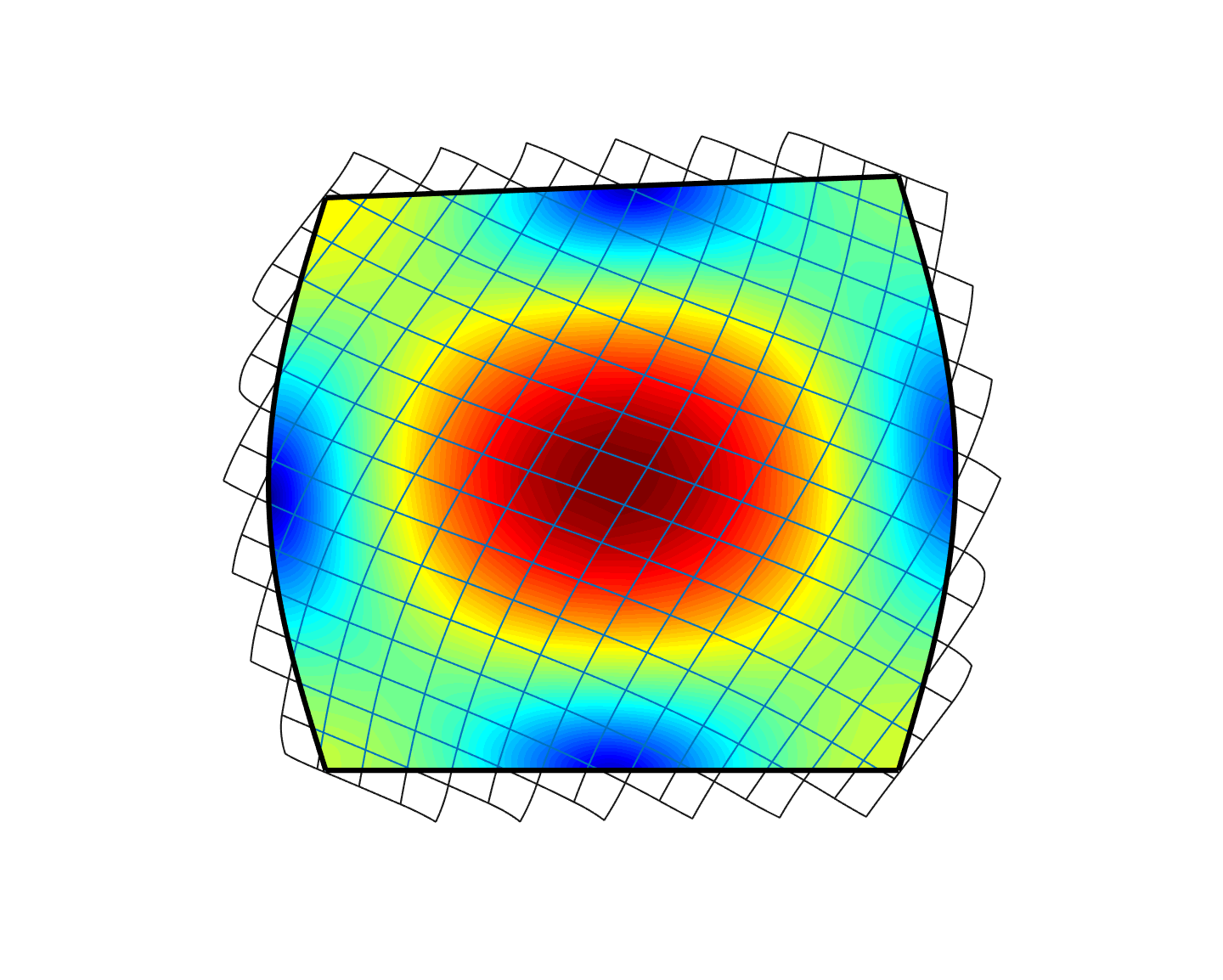}
\caption{Numerical solution}\label{fig:manufactured-sol}
\end{subfigure}
\caption{Example of numerical solution using $C^1Q^2$ splines and mesh size $h=0.1$. The mesh is rotated $\pi/7$ radians to induce cut situations and the removed basis functions are selected using the tolerance $tol=0.01 h^2 \times \sqrt{E}$.}
\label{fig:manufactured-example}
\end{figure}

\begin{figure}
\centering
\begin{subfigure}[t]{.45\linewidth}
\includegraphics[width=\linewidth]{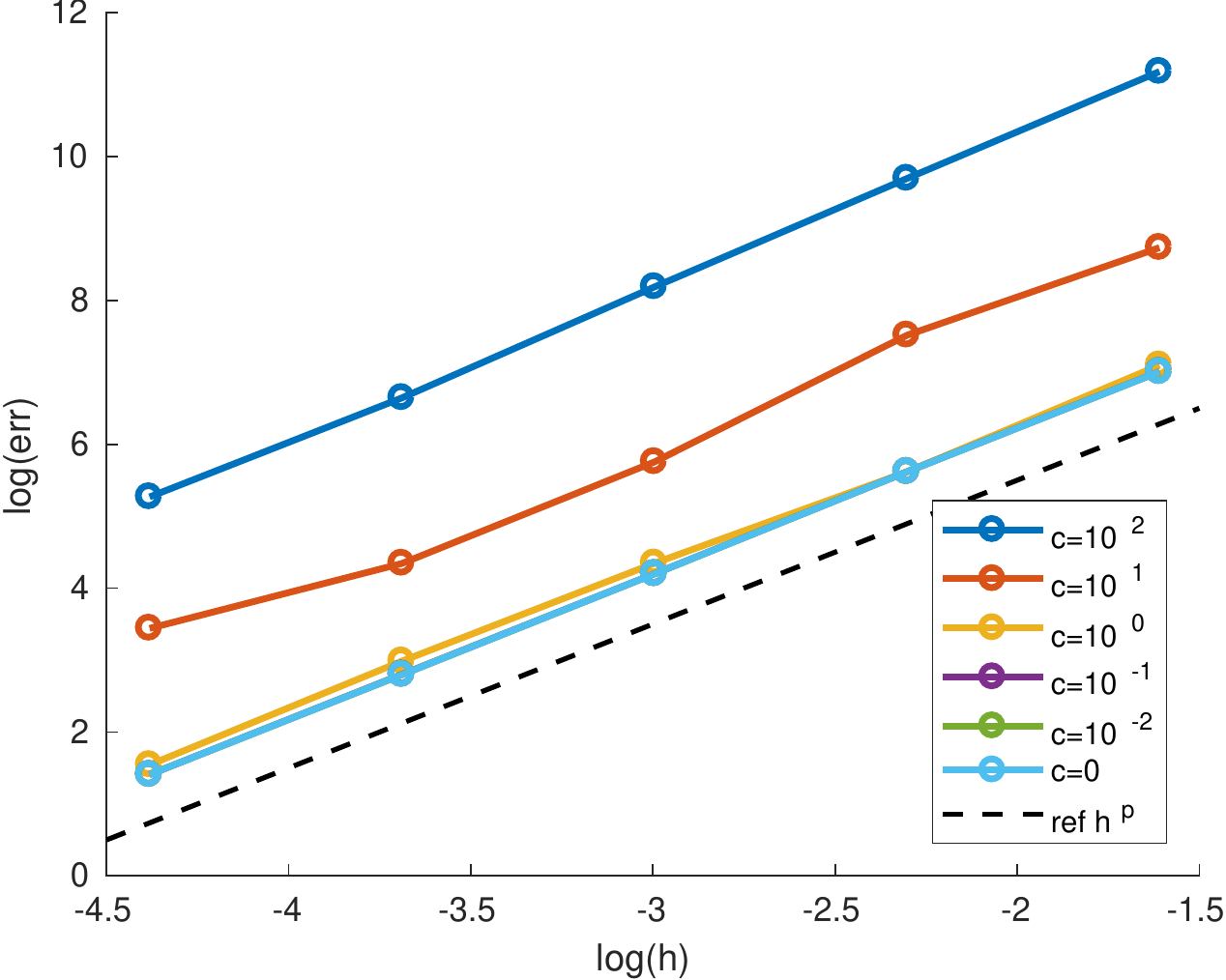}
\caption{Energy norm convergence} \label{fig:conv-energy-norm}
\end{subfigure}\quad
\begin{subfigure}[t]{.45\linewidth}
\includegraphics[width=\linewidth]{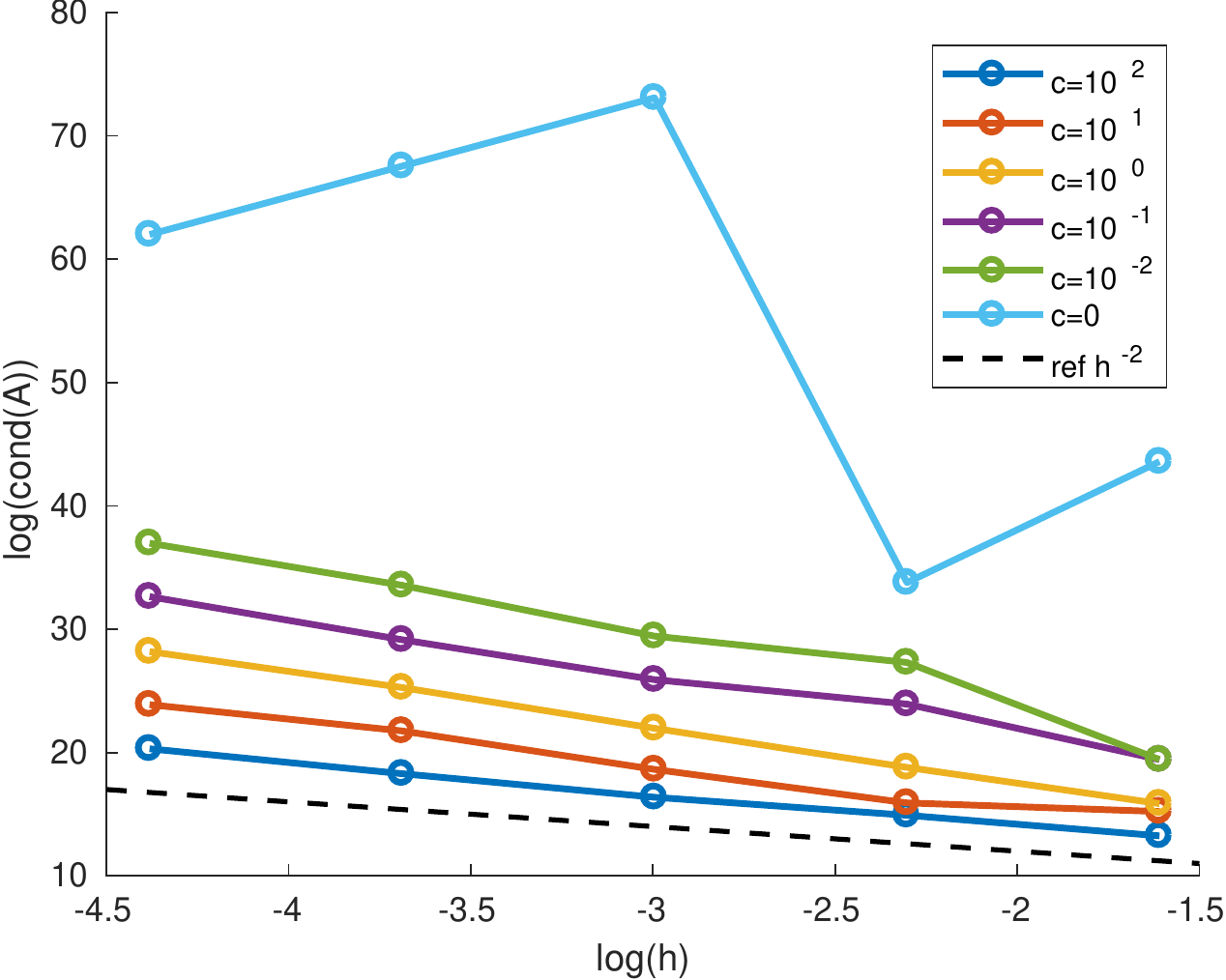}
\caption{Condition number} \label{fig:condest}
\end{subfigure}
\caption{Convergence in $\tn\cdot\tn_h$ norm and condition numbers for the manufactured problem using basis removal with $C^1Q^2$-spline basis. The tolerances used in the selection procedure is $tol = c h^p \times\sqrt{E}$ and we note that the choices $c=10^{-1}$ and $c=10^{-2}$ give no visible difference in the error compared to using the full approximation space ($c=0$).}
\end{figure}

\section{Conclusion} 

We have shown that:
\begin{itemize}
\item Basis function removal can be done in 
a rigorous way which guarantees optimal order of 
convergence and that the resulting linear system is not 
arbitrarily close to singular. These results critically depend 
on the smoothness of the B-spline spaces.
\item Basis function removal is easy to implement and 
efficient since there is no fill-in in the stiffness 
matrix as is the case in for instance face based 
stabilization. Furthermore, basis function 
removal is consistent in contrast to the finite cell 
method.
\end{itemize}
We note however that even though the stiffness matrix is 
not arbitrarily close to singular the resulting condition 
number will in general be worse than $O(h^{-2})$, which is 
the optimal scaling for standard finite element approximation 
of second order elliptic problems and therefore a direct solver 
or preconditioning in combination with an iterative solver 
is necessary in practice.

\bibliographystyle{abbrv}
\bibliography{refs}

\end{document}